\documentclass[reqno]{amsart}

\usepackage[left=26mm,top=29.9mm,right = 26mm, bottom = 29.9mm]{geometry}

\usepackage[english]{babel} 
\usepackage[T1]{fontenc}
\usepackage{tikz-cd}
\usepackage{enumitem}
\usepackage{amsthm}
\usepackage[backgroundcolor=blue!5!white]{todonotes}
\usepackage{mathtools} 
\usepackage[pagebackref]{hyperref}
\usepackage{mathrsfs}
\usepackage{quiver}
\usepackage{tikz}
\usepackage{graphicx}
\usepackage[dvipsnames]{xcolor}
\usepackage{orcidlink}

\makeatletter
\let\@cite@ofmt\@firstofone
\makeatother

\usepackage[capitalize]{cleveref} 

\theoremstyle{plain}

\newtheorem{theorem}{Theorem}[section]
\newtheorem{proposition}[theorem]{Proposition}
\newtheorem{lemma}[theorem]{Lemma}
\newtheorem{corollary}[theorem]{Corollary}

\theoremstyle{definition}

\newtheorem{definition}[theorem]{Definition}
\newtheorem*{definition*}{Definition}
\newtheorem{notation}[theorem]{Notation}
\newtheorem{example}[theorem]{Example}

\newtheorem{remark}[theorem]{Remark}
\newtheorem*{remark*}{Remark}

\AddToHook{env/proposition/begin}{\crefalias{theorem}{proposition}}
\AddToHook{env/lemma/begin}{\crefalias{theorem}{lemma}}
\AddToHook{env/corollary/begin}{\crefalias{theorem}{corollary}}

\AddToHook{env/definition/begin}{\crefalias{theorem}{definition}}
\AddToHook{env/notation/begin}{\crefalias{theorem}{notation}}
\AddToHook{env/example/begin}{\crefalias{theorem}{example}}
\AddToHook{env/examples/begin}{\crefalias{theorem}{examples}}

\crefname{notation}{Notation}{Notations}
\Crefname{notation}{Notation}{Notations}

\AddToHook{env/remark/begin}{\crefalias{theorem}{remark}}

\newcommand{\expl}[2]{%
  \underset{%
    \substack{%
      \big\uparrow\\[-0.2ex]
      \mathrlap{%
        \mbox{%
          \footnotesize
          \hspace{-1em}%
          \begin{tabular}{@{}l@{}}
            #2
          \end{tabular}%
        }%
      }%
    }%
  }{#1}%
}

\newcommand{\leqexpl}[1]{\expl{\leq}{#1}}

\newcommand{\cat}[1]{\mathsf{#1}}

\newcommand{\Set}{\cat{Set}}
\newcommand{\BA}{\cat{BA}}
\newcommand{\Pos}{\cat{Pos}}
\newcommand{\Ctx}{\cat{Ctx}}
\newcommand{\D}{\cat{D}}
\newcommand{\C}{\mathsf{C}}
\newcommand{\FinSet}{\cat{FinSet}}

\newcommand{\LT}{\mathsf{LT}}

\newcommand{\N}{\mathbb{N}}

\newcommand{\T}{\mathcal{T}}

\newcommand{\tmn}{\mathbf{1}}

\newcommand{\F}{\mathbb{F}}

\renewcommand{\P}{\mathbf{P}}
\newcommand{\R}{\mathbf{R}}

\newcommand{\I}{\mathbf{I}}

\newcommand{\ex}[2]{(\exists{#1})_{#2}}
\newcommand{\exn}[3]{(\exists^{#1}{#2})_{#3}}
\newcommand{\fan}[3]{(\forall^{#1}{#2})_{#3}}
\newcommand{\eq}[2]{(\textup{\AE}{#1})_{#2}}

\newcommand{\ple}[1]{\langle#1\rangle}

\newcommand{\pws}{\mathscr{P}}

\DeclareMathOperator{\id}{id}
\DeclareMathOperator{\pr}{pr}
\newcommand{\op}{^{\mathrm{op}}}

\newcommand{\inlinecorner}{%
  \vcenter{\hbox{\scalebox{1}{%
    \begin{tikzcd}[
        ampersand replacement=\&,
        cramped,
        row sep=0.7em,
        column sep=0.7em,
        cells={nodes={inner sep=0.7pt}}
    ]
        \& \bullet \arrow[d] \\
        \bullet \arrow[r] \& \bullet
    \end{tikzcd}%
  }}}%
}

\usetikzlibrary{positioning,fit,calc,decorations.pathreplacing,arrows.meta,shapes.geometric}

\makeatletter
\def\l@subsection{%
  \settowidth{\@tempdimb}{1.\quad}%
  \@tocline{2}{0pt}
    {\dimexpr1pc+\@tempdimb\relax}
    {6pc}{}}
\makeatother

\title{On the Beck--Chevalley condition}

\hypersetup{
    pdftitle={On the Beck--Chevalley condition},
    pdfauthor={Marco Abbadini, Francesca Guffanti},
    pdfmenubar=false,
    pdffitwindow=true,
    pdfstartview={FitH null},
    colorlinks=true,
    linkcolor=teal,
    citecolor=magenta,
    urlcolor=cyan
}

\keywords{Hyperdoctrines, Beck--Chevalley condition}

\subjclass[2020]{Primary: 03G30. Secondary: 03G05, 03B10, 03G15, 06E25}
\author[Marco Abbadini]{Marco Abbadini\textsuperscript{ \orcidlink{0000-0003-1292-6006}}}
\address[Marco Abbadini]{Université catholique de Louvain, Research Institute in Mathematics and Physics, 
Chem.\ du Cyclotron 2, 1348 Louvain-la-Neuve, Belgium}
\urladdr{\url{https://marcoabbadini-uni.github.io}}

\author[Francesca Guffanti]{Francesca Guffanti\textsuperscript{ \orcidlink{0009-0005-8792-0655}}}
\address[Francesca Guffanti]{Université Savoie Mont Blanc,
LAMA, Campus Scientifique,
73376 Le Bourget-du-Lac Cedex, France}
\urladdr{\url{https://sites.google.com/view/francesca-guffanti}}

\begin{document}

\begin{abstract}
    Boolean hyperdoctrines provide an algebraic semantics for classical first-order logic with equality.
    In the definition of a Boolean hyperdoctrine, the Beck--Chevalley condition captures the commutativity of substitutions with quantifiers and with equality by requiring the commutativity of an appropriate square for each pullback square of a specified form in the base category.
    Often, a generalization of this condition is considered, which requires the commutativity of an appropriate square for \emph{every} pullback square in the base category. A Boolean hyperdoctrine satisfying this condition is called \emph{full}.
    
    On the negative side, we exhibit various non-full Boolean hyperdoctrines.
    On the positive side, we give a direct proof that every Boolean hyperdoctrine $\FinSet \to \BA$ over $\FinSet\op$ is full; this means that
    fullness always occurs for the Lindenbaum--Tarski hyperdoctrine of a first-order theory in a one-sorted purely relational language with equality.
    Moreover, we extend this result to the many-sorted case.
\end{abstract}

\maketitle

\tableofcontents

\section{Introduction}

\textbf{\underline{Hyperdoctrines}}.
Hyperdoctrines were introduced as a categorical framework for logic by Lawvere in his seminal works \cite{Lawvere69,Lawvere70}. A \emph{Boolean hyperdoctrine} over a category $\C$ with finite products is a functor $\P\colon\C\op\to\BA$ (where $\BA$ denotes the category of Boolean algebras) satisfying suitable conditions.
A motivating example comes from classical first-order logic (\cref{ex:synt-doc}): an object of $\C$ is a context (i.e.\ a finite set of variables), a morphism $X\to Y$ in $\C$ is an assignment of a term in context $X$ to each variable in $Y$, a fiber $\P(X)$ is the
Lindenbaum--Tarski algebra of formulas in context $X$, and the reindexing along a morphism is the corresponding substitution function. Suitable reindexing maps have left adjoints, which encode existential quantification and equality. The \emph{Beck--Chevalley condition} expresses the compatibility of these adjunctions with substitution.

\textbf{\underline{Beck--Chevalley condition and fullness}}. In the definition of a Boolean hyperdoctrine that we follow, the Beck--Chevalley condition is required only for certain squares: those arising from projections $X \times Y \to X$ (related to quantifiers) and from diagonals-in-context $X \times Y \to X \times Y \times Y$ (related to equality). These squares are always pullback squares.
We say that a Boolean hyperdoctrine is \emph{full} if the Beck--Chevalley condition holds for every pullback square in the base category (\cref{def:full}).
While the presence of existential quantifiers and equality already implies that the reindexing along any morphism has a left adjoint \cite[Rem.~2.13]{MaiettiRosolini2013a}, statements in the literature suggest that the Beck--Chevalley condition for the pullback squares related to quantifiers and equality should not force the Beck--Chevalley condition for all existing pullbacks: Pitts observed that there is no reason for the Beck--Chevalley condition to hold for all pullback squares that happen to exist in the base category \cite[Rem.~4.6]{Pitts1999}, and Jacobs made the same point in the fibrational language, stressing that the Beck--Chevalley condition is an external condition involving pullbacks in the base category \cite[Example~4.3.7]{Jacobs1999}.
Examples lacking fullness are known for elementary pure existential doctrines, which capture the regular-logic fragment $\top,\land,\exists,=$ \cite[Rem.~6.4]{MaiettiTrotta2023}.
However, we have not found in the literature a counterexample in the Boolean setting.

\textbf{\underline{Negative result: examples of non-full Boolean hyperdoctrines}}.
On the negative side, we provide examples of non-full Boolean hyperdoctrines. We give a general procedure: if $\pws\colon\Set\op\to\BA$ denotes the powerset hyperdoctrine (\cref{ex:subset-hyp}) and $F\colon \C\to\Set$ is a finite-product-preserving functor from a category $\C$ with finite products to $\Set$, then the composite
\[
    \C\op \xrightarrow{F\op} \Set\op \xrightarrow{\pws} \BA
\]
is a Boolean hyperdoctrine (\cref{p:recipe}\eqref{i:recipe-1}). Moreover, this hyperdoctrine is full precisely when $F$ preserves the pullbacks that exist in $\C$ (\cref{p:full-powerset-preserves-pullbacks}).
As a consequence, every finite-product-preserving functor $F\colon\C \to \Set$ (i.e., any model of the algebraic theory $\C$) that fails to preserve some pullback yields a non-full Boolean hyperdoctrine $\pws \circ F\op$.
A particularly simple example is obtained from the non-emptiness functor $N\colon\Set\to\Set$, which sends the empty set to itself and every non-empty set to a singleton. This functor preserves finite products but not pullbacks.
The corresponding composite
\[
    \Set\op \xrightarrow{N\op} \Set\op \xrightarrow{\pws} \BA
\]
is a non-full Boolean hyperdoctrine (\cref{t:BC-pi0-counterexample}).
Concretely, it is the functor $\Set\op \to \BA$ that assigns the one-element Boolean algebra to the empty set and the two-element Boolean algebra to each non-empty set.
We then give a second example of a non-full Boolean hyperdoctrine $\C\op \to \BA$, whose base category $\C$ is the category of contexts of a first-order language with a single unary function symbol (\cref{ex:non-full-contexts}): also here the hyperdoctrine is obtained from a finite-product-preserving functor $\C \to \Set$ that fails to preserve some pullback.

\textbf{\underline{Positive result: relational languages}}.
On the positive side, we first show that every Boolean hyperdoctrine
$\FinSet\to\BA$ over $\FinSet\op$ is full
(\cref{t:all-full-finset}): we give a direct proof by decomposing each pushout in $\FinSet$
into a pasting of simpler pushout squares, each of which satisfies the Beck--Chevalley condition.\footnote{The fact that every Boolean hyperdoctrine $\FinSet\to\BA$ over $\FinSet\op$ is full can also be deduced indirectly from Kamsma's result on the Beck--Chevalley condition for type space functors of one-sorted coherent theories \cite[Prop.~4.11]{Kamsma2023}. See the introduction of \cref{s:bool-doct-finset} for more details.}
We then extend this argument to finite sets equipped with
sorts (\cref{p:all-full-many-sorted}).
This means that the Lindenbaum--Tarski hyperdoctrine of every first-order theory in a many-sorted purely relational language with equality is full (\cref{c:syn-full}).

\textbf{\underline{Structure of the paper}}. The paper is organized as follows. 
\begin{itemize}
    \item In \cref{s:bool-hyp}, we recall the definition of a Boolean hyperdoctrine and of a full Boolean hyperdoctrine.
    \item In \cref{s:non-full}, we exhibit non-full Boolean hyperdoctrines.
    \item In \cref{s:bool-doct-finset}, we prove that every Boolean hyperdoctrine over $\FinSet\op$ is full and extend this result to the many-sorted setting.
    \item In Appendix~\ref{appendix}, we collect equivalent formulations of having equality for a Boolean-valued functor, which are used in \cref{s:bool-doct-finset}.
    \item
    In Appendix~\ref{appendix-B}, we discuss the relation with Seely's paper \cite{Seely83}.
    Indeed, Seely characterized the pullback squares satisfying the Beck--Chevalley condition in terms of the derivability of two sequents in the associated theory \cite[\S 8]{Seely83}. Seely's notion of hyperdoctrine, however, is not the one adopted here: in particular, it requires that, whenever a pullback square satisfies the Beck--Chevalley condition, so does the square obtained by multiplying it by a fixed object. We show that this requirement does not follow from the axioms of a Boolean hyperdoctrine, so that Seely's characterization is not available for an arbitrary Boolean hyperdoctrine, and in particular cannot be applied directly to obtain fullness of every Boolean hyperdoctrine over $\FinSet\op$ (\cref{t:all-full-finset}).
\end{itemize}

\section{Preliminaries: Boolean hyperdoctrines}\label{s:bool-hyp}

\begin{notation}[Standing notation]\hfill
    \begin{enumerate}
        \item 
        $\N$ denotes the set of natural numbers, including $0$.
        
        \item
        $\BA$ denotes the category of Boolean algebras and Boolean homomorphisms.
        \item $\Pos$ denotes the category of partially ordered sets and order-preserving functions.
        
        \item
        $\tmn_\C$ (or simply $\tmn$) denotes the terminal object (when it exists) of a category $\C$.
    \end{enumerate}
\end{notation}

\subsection{Boolean hyperdoctrines}

The notion of a hyperdoctrine has its roots in the work of Lawvere \cite{Lawvere69,Lawvere70}; the following definition of a \emph{Boolean hyperdoctrine} captures classical first-order logic with equality\footnote{The doctrinal version of ``having equality'' that we use is in the arXiv preprint \cite[Def.~2.5]{MaiettiArxiv}. Note that, in the published version \cite[Def.~2.5]{MaiettiRosolini2013a}, the text appears unintentionally truncated and missing the Beck--Chevalley condition.}.

\begin{definition}[Boolean hyperdoctrine]\label{d:bool_hyp}
    Given a category $\C$ with finite products, a \emph{Boolean hyperdoctrine over $\C$} is a functor $\P \colon \C\op \to  \BA$ with the following properties.
    \begin{enumerate}
        \item \label{i:h1} {(Existential)} 
        For all $X, Y \in \C$, letting $\pr^{X\times Y}_X \colon X \times Y \to X$ denote the projection onto the first coordinate, the function
        \[
        \P(\pr^{X\times Y}_X) \colon \P(X) \longrightarrow \P(X \times Y)
        \]
        has a left adjoint, denoted $\ex{Y}{X}$ (which is not required to be a Boolean homomorphism).
        
        \item \label{i:h2}
        (Beck--Chevalley for existential) For any morphism $f\colon X'\to X$ in $\C$ and every $Y \in \C$, the following square in $\Pos$ commutes. 
        \[
            \begin{tikzcd}
            {X} & {\P(X\times Y)} & {\P(X)} \\
            X' & {\P(X'\times Y)} & {\P(X')}
            \arrow["{\P(f\times\id_{Y})}", from=1-2, to=2-2, swap]
            \arrow["{\P(f)}", from=1-3, to=2-3]
            \arrow["{\ex{Y}{X'}}"', from=2-2, to=2-3]
            \arrow["{\ex{Y}{X}}", from=1-2, to=1-3]
            \arrow["f", from=2-1, to=1-1]
            \end{tikzcd}
        \]
        
        \item\label{i:h3} (Equality) For all $X,Y\in\C$, letting $\Delta_Y \colon Y \to Y \times Y$ denote the diagonal morphism $\ple{\id_Y,\id_Y}$, the function 
        \[
        \P(\id_X\times \Delta_Y)\colon \P(X\times Y\times Y)\to \P(X\times Y)
        \]
        has a left adjoint, denoted $\eq{Y}{X}$ (which is not required to be a Boolean homomorphism).
        
         \item\label{i:h4}
        (Beck--Chevalley for equality) For any morphism $f\colon X'\to X$ in $\C$ and every $Y \in \C$, the following square in $\Pos$ commutes. 
        \[
            \begin{tikzcd}
            {X} & {\P(X\times Y)} & {\P(X\times Y\times Y)} \\
            X' & {\P(X'\times Y)} & {\P(X'\times Y\times Y)}
            \arrow["{\P(f\times\id_{Y})}", from=1-2, to=2-2, swap]
            \arrow["{\P(f\times \id_{Y\times Y})}", from=1-3, to=2-3]
            \arrow["{\eq{Y}{X'}}"', from=2-2, to=2-3]
            \arrow["{\eq{Y}{X}}", from=1-2, to=1-3]
            \arrow["f", from=2-1, to=1-1]
            \end{tikzcd}
        \]
        
    \end{enumerate}
\end{definition}

The category $\C$ is called the \emph{base category of $\P$}.
For $X\in \C$, $\P(X)$ is called the \emph{fiber} over $X$.
For a morphism $f\colon X'\to X$, the function $\P(f) \colon\P(X)\to \P(X')$ is called the \emph{reindexing along $f$}.

\begin{example}[Syntactic hyperdoctrine]\label{ex:synt-doc}
    Let $\T$ be a theory in a one-sorted first-order language with equality, and denote by $\F$ the set of function symbols. 
    The \emph{syntactic hyperdoctrine of $\T$} (or Lindenbaum--Tarski hyperdoctrine of $\T$) is the Boolean hyperdoctrine
    \[
    \LT^{\T} \colon \Ctx_\F \op\longrightarrow \BA
    \]
    (where $\LT$ stands for ``Lindenbaum--Tarski algebra'') defined as follows. 
    \begin{itemize}
        
        \item An object of the base category $\Ctx_\F$ is a finite set of variables (also called a \emph{context}).
        
        \item A morphism $X\to Y$ in $\Ctx_\F$ is a function
        \[
        \sigma\colon Y\longrightarrow\mathrm{Term}_{\F}(X),
        \]
        where $\mathrm{Term}_{\F}(X)$ denotes the set of terms in context $X$.
        The identity on $X$ is the inclusion $X\hookrightarrow \mathrm{Term}_\F(X)$.
        The composition of morphisms is given by simultaneous substitutions: given $\sigma \colon X \to Y$ and $\tau \colon Y \to Z$, the composite $\tau\circ\sigma \colon X \to Z$ is the function $z \mapsto \tau(z)[\sigma(y)/y]_{y \in Y}$.
        
        \item In $\Ctx_\F$, the terminal object is the empty set $\varnothing$, and the product of two objects $X$ and $Y$ is the disjoint union $X\sqcup Y$.
        
        \item On objects, $\LT^{\T} \colon \Ctx_\F\op\to \BA$ maps a context $X$ to the poset reflection of the preordered set of formulas whose free variables belong to $X$, ordered by provable consequence $\vdash_\T$ in $\T$, according to which $\alpha$ is below $\beta$ if and only if the sequent $\alpha \Rightarrow_X \beta$ is provable from $\T$; here, the subscript $X$ in the sequent symbol $\Rightarrow$ means that the sequent is considered in the context $X$.\footnote{We refer for example to \cite[Appendix~A]{AbbadiniGuffanti} for the rules of the sequent calculus with contexts for classical first-order logic. A consequence of the slight difference between the calculus \emph{with} contexts and the usual calculus \emph{without} contexts is that the sequent $\Rightarrow_{\varnothing} \exists x \top$ is in general not provable in the former, in accordance with admitting the empty set as a possible model.}
        
        \item On morphisms, $\LT^{\T}$ maps $\sigma\colon X\to Y$ to the substitution $[\sigma(y)/y]_{y\in Y} \colon \LT^{\T}(Y) \to \LT^{\T}(X)$.
        
        \item Let $X$ and $Y$ be finite sets of variables,
        and let $\pr_X$ denote the projection morphism $X\sqcup Y\to X$ in $\Ctx_\F$. 
        The Boolean homomorphism $\LT^{\T}(\pr_X)\colon\LT^{\T}(X)\to\LT^{\T}(X\sqcup Y)$, which maps a formula to itself but with the variables in $Y$ considered as dummy variables,
        has as left adjoint the function
        \[
        \exists y_1\dots\exists y_m\colon\LT^{\T}(X\sqcup Y)\longrightarrow\LT^{\T}(X),
        \]
        where $y_1,\dots,y_m$ is any enumeration of the elements of $Y$.
        
        \item Let $X$ and $Y$ be finite sets of variables and let $Y' = \{y'_1, \dots, y'_m\}$ be a copy of $Y = \{y_1,\dots, y_m\}$.
        The Boolean homomorphism $\LT^\T(\id_X\times \Delta_Y)\colon\LT^\T(X \sqcup Y \sqcup Y') \to \LT^\T(X \sqcup Y)$, which maps $\alpha$ to $\alpha[y_i/y'_i]_{i=1}^m$, has as left adjoint the function
            \begin{align*}
                \LT^\T(X \sqcup Y) &\longrightarrow \LT^\T(X \sqcup Y \sqcup Y')\\
                \beta &\longmapsto
                        \beta \wedge (y_1 = y'_1) \wedge \cdots\wedge (y_m = y'_m).
            \end{align*}
    \end{itemize}
\end{example}

The two structures (existential and equality) in the definition of a Boolean hyperdoctrine are very similar, and in fact they can be unified as follows.

\begin{definition}[Boolean hyperdoctrine, compact definition] \label{d:bool_hyp_alt}
    Given a category $\C$ with finite products, a \emph{Boolean hyperdoctrine over $\C$} is a functor $\P \colon \C\op \to  \BA$ with the following properties.
    \begin{enumerate}
        \item  \label{i:ha1} {(Adjoints)}
        For all $X, Y \in \C$ and $n \in \N$, letting $\Delta^n_Y \colon Y \to Y^n$ denote the unique morphism whose composite with each projection gives the identity on $Y$, the function
        \[
        \P(\id_X \times \Delta^n_Y) \colon \P(X \times Y^n) \longrightarrow \P(X \times Y)
        \]
        has a left adjoint, denoted $\exn{n}{Y}{X}$.

        \item \label{i:ha2} (Beck--Chevalley) For every morphism $f\colon X'\to X$ in $\C$, every $Y \in \C$ and every $n \in \N$, the following square in $\Pos$ commutes. 
        \[
            \begin{tikzcd}
            {X} & {\P(X\times Y)} & {\P(X\times Y^n)} \\
            X' & {\P(X'\times Y)} & {\P(X'\times Y^n)}
            \arrow["{\P(f\times\id_{Y})}", from=1-2, to=2-2, swap]
            \arrow["{\P(f \times \id_{ Y^n})}", from=1-3, to=2-3]
            \arrow["{\exn{n}{Y}{X'}}"', from=2-2, to=2-3]
            \arrow["{\exn{n}{Y}{X}}", from=1-2, to=1-3]
            \arrow["f", from=2-1, to=1-1]
            \end{tikzcd}
        \]
    \end{enumerate}
\end{definition}

\begin{proposition}[Equivalence of the two definitions]
    \cref{d:bool_hyp} and \cref{d:bool_hyp_alt} are equivalent.
\end{proposition}

\begin{proof}
    The conditions in \cref{d:bool_hyp_alt} imply those in \cref{d:bool_hyp}, since the ones in \cref{d:bool_hyp} are the cases $n = 0$ and $n = 2$ of the ones in \cref{d:bool_hyp_alt}.
    
    Conversely, suppose that $\P$ is a Boolean hyperdoctrine according to \cref{d:bool_hyp}. We show by induction on $n\in\N$ that for all $X,Y$ in $\C$ the function $\P(\id_X \times \Delta^n_Y)$ has a left adjoint $\exn{n}{Y}{X}$ with the Beck--Chevalley condition.
    The cases $n=0,2$ hold by assumption. Moreover, the case $n=1$ is trivial since $\Delta^1_Y$ is $\id_Y$: indeed, $\P(\id_{X \times Y}) = \id_{\P(X \times Y)}$, and the identity function trivially has itself as a left adjoint, and the Beck--Chevalley square clearly commutes.
    
    Now let us consider $n>2$. The morphism $\id_X\times \Delta^n_Y \colon X\times Y \to X\times Y^n$ is the composite
    \[
    X\times Y\xrightarrow{\id_X\times \Delta^2_Y} X\times Y\times Y\xrightarrow{\id_{X\times Y}\times \Delta^{n-1}_Y}X\times Y\times Y^{n-1}.
    \]
    By the inductive hypothesis (treating $X \times Y$ as the new base object), the map $\P(\id_{X \times Y} \times \Delta^{n-1}_Y)$ has a left adjoint $\exn{n-1}{Y}{X\times Y}$ satisfying the Beck--Chevalley condition. By the case $n = 2$, also the map $\P(\id_X \times \Delta^2_Y)$ has a left adjoint $\exn{2}{Y}{X}$ satisfying the Beck--Chevalley condition.
    Since the composite of two left adjoints is the left adjoint of the composite, the left adjoint $\exn{n}{Y}{X}$ is given by $\exn{n-1}{Y}{X\times Y}\circ\exn{2}{Y}{X}$. Finally, the Beck--Chevalley condition holds by pasting two commuting Beck--Chevalley squares.
    \[
        \begin{tikzcd}[row sep=0.8cm, column sep=1.4cm,baseline=(\tikzcdmatrixname-3-1.base)]
        	{\P(X\times Y)} & {\P(X\times Y\times Y)} & {\P(X\times Y\times Y^{n-1})} \\
        	\\
        	{\P(X'\times Y)} & {\P(X'\times Y\times Y)} & {\P(X'\times Y\times Y^{n-1})}
        	\arrow["{\exn{2}{Y}{X}}", from=1-1, to=1-2]
        	\arrow["{\P(f\times \id_Y)}"', from=1-1, to=3-1]
        	\arrow["{\exn{n-1}{Y}{X\times Y}}", from=1-2, to=1-3]
        	\arrow["\begin{array}{c} \P(f\times \id_{Y^2})\\=\P(f\times \id_Y\times \id_Y) \end{array}"{description}, from=1-2, to=3-2]
        	\arrow["\begin{array}{c} \P(f\times \id_{Y^n} )\\=\P(f\times \id_Y\times \id_{Y^{n-1}}) \end{array}"{description}, from=1-3, to=3-3]
        	\arrow["{\exn{2}{Y}{X'}}"', from=3-1, to=3-2]
        	\arrow["{\exn{n-1}{Y}{X'\times Y}}"', from=3-2, to=3-3]
        \end{tikzcd}
        \qedhere
    \]
\end{proof}

\begin{remark}[Existentiality $\Leftrightarrow$ universality] \label{r:ex-iff-un}
    In \cref{d:bool_hyp_alt}, we required the existence of the left adjoint $\exn{n}{Y}{X} \colon \P(X \times Y)\to\P(X \times Y^n)$ of $\P(\id_X \times \Delta^n_Y)$ (with the Beck--Chevalley condition) for all $X,Y\in\C$ and $n \in \N$.
    In this Boolean case, this condition is equivalent to the existence of the right adjoint $\fan{n}{Y}{X}$ (with the Beck--Chevalley condition), as the two quantifiers are interdefinable: $\forall = \lnot \exists \lnot$ and $\exists = \lnot \forall \lnot$.
\end{remark}

\subsection{Full Boolean hyperdoctrines}

\begin{notation}[Equality predicate]\label{n:delta}
    Given a Boolean hyperdoctrine $\P$ over $\C$ and $Y \in \C$, we set
    \[
    \delta_Y\coloneqq\eq{Y}{\tmn}\top_{\P(Y)}.
    \]
    We refer to \cref{t:defs_equality}\eqref{i_eq:3}
    for further details about the properties of the family $(\delta_Y)_{Y\in\C}$.
\end{notation}

\begin{remark}[Adjoint to arbitrary reindexing]\label{r:every-left-adj}
    It is well known that, for every Boolean hyperdoctrine $\P \colon \C\op \to \BA$, the reindexing along \emph{any} morphism $f\colon X\to Y$ in $\C$ has a left adjoint $\exists_f\colon \P(X)\to \P(Y)$, namely
        \begin{align*}
            \exists_f\colon \P(X)&\longrightarrow \P(Y)\\
            \alpha&\longmapsto\ex{X}{Y}(\P(\id_Y \times f)(\delta_Y) \land \P(\pr^{Y \times X}_X)(\alpha));
        \end{align*}
    see \cite[Rem.~2.13]{MaiettiRosolini2013a}, and see also \cite[Prop.~2.6]{MaiPaRo} for the precise form given here; cf.\ also \cite[Example~4.3.7]{Jacobs1999}.
\end{remark}

\begin{definition}
    Let $\C$ be a category, let $\P \colon \C\op \to \BA$ be a functor and let
    \begin{equation} \label{eq:def-bc}
        \begin{tikzcd}
        	W & Y\\
        	X & Z 
        	\arrow["v", from=1-1, to=1-2]
        	\arrow["u"', from=1-1, to=2-1]
        	\arrow["g", from=1-2, to=2-2]
        	\arrow["f"', from=2-1, to=2-2]
        \end{tikzcd} 	
    \end{equation}
    be a commutative square in $\C$. Suppose that the Boolean homomorphisms $\P(f)\colon \P(Z)\to \P(X)$ and $\P(v)\colon \P(Y)\to \P(W)$ have left adjoints $\exists_f$ and $\exists_v$ respectively. We say that \emph{the square \eqref{eq:def-bc} satisfies the Beck--Chevalley condition} (relative to $\P$) if applying $\P$ to the commutative square \eqref{eq:def-bc} and then taking the left adjoints to the reindexings of the horizontal maps gives a commutative square, i.e., if the following diagram commutes.
    \[\begin{tikzcd}
        	{\P(X)} & {\P(Z)} \\
        	{\P(W)} & {\P(Y)}
        	\arrow["{\exists_f}", from=1-1, to=1-2]
        	\arrow["{\P(u)}"', from=1-1, to=2-1]
        	\arrow["{\P(g)}", from=1-2, to=2-2]
        	\arrow["{\exists_v}"', from=2-1, to=2-2]
        \end{tikzcd}\]
\end{definition}

\begin{remark}[Standard Beck--Chevalley squares are pullbacks]\label{r:pullback}
    Given a category $\C$ with finite products, a morphism $f \colon X' \to X$ in $\C$, $Y \in \C$, and $n \in \N$, we have the following commutative diagram.
    \begin{equation} \label{eq:diagram}
        \begin{tikzcd}[column sep = large]
            	X'\times Y \arrow[r,"{\id_{X'}\times \Delta^n_Y}"]\arrow[d,"{f\times \id_{Y}}"'] & X' \times Y^n\arrow[d,"{f\times \id_{Y^n}}"]  \\
            	X \times Y\arrow[r,"{\id_{X}\times \Delta^n_Y}"'] & X \times Y^n 
        \end{tikzcd}   	
    \end{equation}
    The Beck--Chevalley condition for a Boolean hyperdoctrine $\P \colon \C\op \to \BA$ (\cref{d:bool_hyp_alt}\eqref{i:ha2}) says that the square \eqref{eq:diagram} satisfies the Beck--Chevalley condition.
    It is easily seen that the square \eqref{eq:diagram} is a pullback.
\end{remark}

This leads to the following definition.

\begin{definition}[Full Boolean hyperdoctrine]\label{def:full}
    Given a category $\C$ with finite products, a \emph{full Boolean hyperdoctrine over $\C$} is a functor $\P \colon \C\op \to  \BA$ with the following properties.
    \begin{enumerate}
        \item \label{i:full-1} {(Full existential)} 
        For every morphism $f\colon X\to Y$ in $\C$, the function
        \[
        \P(f) \colon \P(Y) \longrightarrow \P(X)
        \]
        has a left adjoint $\exists_f$.
        
        \item (Full Beck--Chevalley) Every pullback square in $\C$ satisfies the Beck--Chevalley condition. That is, for any pullback square in $\C$ (on the left), the square in $\Pos$ on the right commutes.
        \[\begin{tikzcd}
        	W & Y & {\P(X)} & {\P(Z)} \\
        	X & Z & {\P(W)} & {\P(Y)}
        	\arrow["v", from=1-1, to=1-2]
        	\arrow["u"', from=1-1, to=2-1]
        	\arrow["\lrcorner"{anchor=center, pos=0.125}, draw=none, from=1-1, to=2-2]
        	\arrow["g", from=1-2, to=2-2]
        	\arrow["{\exists_f}", from=1-3, to=1-4]
        	\arrow["{\P(u)}"', from=1-3, to=2-3]
        	\arrow["{\P(g)}", from=1-4, to=2-4]
        	\arrow["f"', from=2-1, to=2-2]
        	\arrow["{\exists_v}"', from=2-3, to=2-4]
        \end{tikzcd}\]
    \end{enumerate}
\end{definition}

The name \emph{full} appears in \cite[Def.~3.5]{MaiettiTrotta2023}.

\begin{remark}[The automatic Beck--Chevalley inequality]\label{r:leq}
    Let us observe that, for every functor $\P\colon \C\op\to\BA$ and every commutative square
    \begin{equation}\label{diag:leq}
        \begin{tikzcd}
        	W & Y\\
        	X & Z 
        	\arrow["v", from=1-1, to=1-2]
        	\arrow["u"', from=1-1, to=2-1]
        	\arrow["g", from=1-2, to=2-2]
        	\arrow["f"', from=2-1, to=2-2]
        \end{tikzcd}
    \end{equation}
        such that $\P(f)$ and $\P(v)$ have left adjoints ($\exists_f$ and $\exists_v$ respectively), the inequality
        \[
        \exists_v \circ \P(u) \leq \P(g)\circ \exists_f
        \]
        holds. Indeed, for every $\alpha\in\P(X)$, we have
        \[
        \alpha\leqexpl{unit of $\exists_f\dashv \P(f)$} \P(f)( \exists_f \alpha),
        \]
        and thus, by monotonicity of $\P(u)$, \begin{equation}\label{eq:remark_leq}
            \P(u)(\alpha) \leq \P(u)\big(\P(f)( \exists_f \alpha)\big),
        \end{equation}
        and thus
        \begin{align*}
            &\exists_v  \P(u)(\alpha) \leq \P(g)( \exists_f \alpha)\\
            &\iff   \P(u)(\alpha) \leq \P(v)\big(\P(g)( \exists_f \alpha)\big) &&\text{(since $\exists_v\dashv \P(v)$)}\\
            &\iff   \P(u)(\alpha) \leq \P(u)\big(\P(f)( \exists_f \alpha)\big) &&\text{(by functoriality of $\P$ and commutativity of \eqref{diag:leq})},
        \end{align*}
        which holds by \eqref{eq:remark_leq}, as desired.
\end{remark}

\begin{example}[The powerset hyperdoctrine]\label{ex:subset-hyp}
    The \emph{powerset hyperdoctrine} is the contravariant power set functor $\pws \colon \Set\op \to \BA$, which maps a set $X$ to its power set $\pws(X)$, and a function $f\colon X\to Y$ to the preimage function
        \[
        \pws(f) \coloneqq f^{-1}[-] \colon \pws(Y) \to \pws(X).
        \]

    It is a full Boolean hyperdoctrine. Indeed, for every $f\colon X\to Y$, the direct image function $f[-]\colon \pws(X)\to\pws(Y)$ is the left adjoint to the preimage function $f^{-1}[-]=\pws(f)\colon\pws(Y)\to\pws(X)$. The full Beck--Chevalley condition holds because, for every pullback square in $\Set$ 
    \begin{equation}\label{diag:pws}
    \begin{tikzcd}
	W & Y \\
	X & Z
	\arrow["v", from=1-1, to=1-2]
	\arrow["u"', from=1-1, to=2-1]
	\arrow["\lrcorner"{anchor=center, pos=0.125}, draw=none, from=1-1, to=2-2]
	\arrow["g", from=1-2, to=2-2]
	\arrow["f"', from=2-1, to=2-2]
\end{tikzcd}\end{equation}
and every $A\in\pws(X)$,
\begin{equation}\label{eq:ex-bcc}
g^{-1}[f[A]]=v[u^{-1}[A]].
\end{equation}
(Indeed, the direction $(\supseteq)$ holds for every commutative square---see \cref{r:leq}, and $(\subseteq)$ is seen to hold using the fact that \eqref{diag:pws} is a pullback.)
\end{example}

Commutative squares satisfying the Beck--Chevalley condition relative to $\pws$ admit a simple characterization.
Recall that, in a category with pullbacks, a \emph{quasi-pullback} is a commutative square such that the canonical morphism from the vertex of the square to the vertex of the pullback square is an epimorphism.
In $\Set$, a commutative square
\[\begin{tikzcd}
	W & Y \\
	X & Z
	\arrow["v", from=1-1, to=1-2]
	\arrow["u"', from=1-1, to=2-1]
	\arrow["g", from=1-2, to=2-2]
	\arrow["f"', from=2-1, to=2-2]
\end{tikzcd}\]
is a quasi-pullback if and only if, for all $x\in X$, $y\in Y$ with $f(x)=g(y)$, there is $w\in W$ such that $u(w)=x$ and $v(w)=y$. The following characterizes squares with the Beck--Chevalley condition relative to $\pws$.

\begin{lemma}[Beck--Chevalley if and only if quasi-pullback]\label{l:BCC-set}
    A commutative square in $\Set$ satisfies the Beck--Chevalley condition relative to the powerset hyperdoctrine $\pws \colon \Set\op \to \BA$ if and only if it is a quasi-pullback square.
\end{lemma}
\begin{proof}
    Let 
    \begin{equation}\label{diag:pws-lemma}
    \begin{tikzcd}
	W & Y \\
	X & Z
	\arrow["v", from=1-1, to=1-2]
	\arrow["u"', from=1-1, to=2-1]
	\arrow["g", from=1-2, to=2-2]
	\arrow["f"', from=2-1, to=2-2]
\end{tikzcd}\end{equation}
be a commutative square in $\Set$. We shall prove that it is a quasi-pullback if and only if, for all $A\in\pws(X)$,
\begin{equation}\label{eq:ex-bcc-lemma}
g^{-1}[f[A]]=v[u^{-1}[A]].
\end{equation}
Suppose \eqref{diag:pws-lemma} is a quasi-pullback. Then \eqref{eq:ex-bcc-lemma} follows easily as in \cref{ex:subset-hyp}.

Conversely, suppose \eqref{eq:ex-bcc-lemma} holds for every $A\in\pws(X)$. Let $x\in X$, $y\in Y$ with $f(x)=g(y)$.
In particular, $y\in g^{-1}[f[\{x\}]]$, so that, by \eqref{eq:ex-bcc-lemma}, there is $w\in u^{-1}[\{x\}]$ such that $v(w)=y$, i.e.\ there is $w\in W$ such that $u(w)=x$ and $v(w)=y$, as desired.
\end{proof}

Clearly,
\[
\text{full Boolean hyperdoctrine } \Rightarrow  \text{ Boolean hyperdoctrine}
\]
(see \cref{r:pullback}).
The next section shows that the converse fails.

\section{Non-full Boolean hyperdoctrines}
\label{s:non-full}

\subsection{A recipe for building non-full Boolean hyperdoctrines}

\begin{remark}[Any functor preserves and reflects BC]\label{r:bc-F}
    Let $\C,\D$ be categories, let $F\colon\C\to\D$, $\R\colon\D\op\to\BA$ be functors, let
    \begin{equation}\label{eq:lm-sq}
    \begin{tikzcd}
	W & Y \\
	X & Z
	\arrow["v", from=1-1, to=1-2]
	\arrow["u"', from=1-1, to=2-1]
	\arrow["g", from=1-2, to=2-2]
	\arrow["f"', from=2-1, to=2-2]
    \end{tikzcd}\end{equation}
    be a commutative square in $\C$ and let
    \begin{equation}\label{eq:lm-sq2}
    \begin{tikzcd}
	FW & FY \\
	FX & FZ
	\arrow["Fv", from=1-1, to=1-2]
	\arrow["Fu"', from=1-1, to=2-1]
	\arrow["Fg", from=1-2, to=2-2]
	\arrow["Ff"', from=2-1, to=2-2]
    \end{tikzcd}\end{equation}
    be the image in $\D$ of \eqref{eq:lm-sq} under $F$.
    Suppose that the Boolean homomorphisms $\R(Ff)\colon \R(FZ)\to \R(FX)$ and $\R(Fv)\colon \R(FY)\to \R(FW)$ have left adjoints $\exists_{Ff}$ and $\exists_{Fv}$ respectively.
    The Beck--Chevalley condition for \eqref{eq:lm-sq} relative to $\R\circ F\op$ and the Beck--Chevalley condition for \eqref{eq:lm-sq2} relative to $\R$ are the same condition.
\end{remark}

\cref{p:recipe,p:full-powerset-preserves-pullbacks} below provide a recipe for building non-full Boolean hyperdoctrines.

\begin{proposition}\label{p:recipe}
Let $F \colon \C \to \D$ be a functor between categories with finite products.
\begin{enumerate}
    \item\label{i:recipe-1} If $\R \colon \D\op \to \BA$ is a Boolean hyperdoctrine and $F$ preserves finite products, then $\R \circ F\op$ is a Boolean hyperdoctrine.
    
    \item \label{i:recipe-2} If $\R \colon \D\op \to \BA$ is a full Boolean hyperdoctrine and $F$ preserves pullbacks, then $\R \circ F\op$ is a full Boolean hyperdoctrine.
\end{enumerate}
\end{proposition}

\begin{proof}
    \eqref{i:recipe-1} Let $\P$ denote the composite $\R \circ F\op$.
    Let $X,Y\in\C$ and $n\in\N$. Since $F$ preserves finite products, we have $F(X\times Y^n)\cong FX\times (FY)^n$ and  $F(\id_X\times\Delta^n_Y)=\id_{FX}\times\Delta^n_{FY}$, and hence
    \[
    \P(\id_X\times\Delta^n_Y)=\R(\id_{FX}\times\Delta^n_{FY}),
    \]
    which, by assumption, has a left adjoint, namely $\exn{n}{FY}{FX}$.

    Let us prove the Beck--Chevalley condition: let $f\colon X'\to X$ be a morphism of $\C$. Since $F$ preserves finite products, the image under $F$ of the square on the left is, up to canonical isomorphisms, the square on the right.
    \[
    \begin{tikzcd}[column sep=1.4cm]
        X'\times Y \arrow[r,"{\id_{X'}\times \Delta^n_Y}"]\arrow[d,"{f\times \id_{Y}}"'] & X' \times Y^n\arrow[d,"{f\times \id_{Y^n}}"]  & FX'\times FY \arrow[r,"{\id_{FX'}\times \Delta^n_{FY}}"]\arrow[d,"{Ff\times \id_{FY}}"'] & FX' \times (FY)^n\arrow[d,"{Ff\times \id_{(FY)^n}}"]  \\
        X \times Y\arrow[r,"{\id_{X}\times \Delta^n_Y}"'] & X \times Y^n
        & FX \times FY\arrow[r,"{\id_{FX}\times \Delta^n_{FY}}"'] & FX \times (FY)^n
    \end{tikzcd}
    \]
    The square on the right satisfies the Beck--Chevalley condition relative to $\R$: this is precisely the Beck--Chevalley condition for $\R$ with respect to $Ff$, $FY$ and $n$ (\cref{d:bool_hyp_alt}\eqref{i:ha2}). By \cref{r:bc-F}, the square on the left then satisfies it relative to $\P$, as desired.
    
    \eqref{i:recipe-2} Let $\P$ denote the composite $\R \circ F\op$. Since $\R$ is full, for every morphism $f$ in $\C$, the reindexing $\P(f)=\R(Ff)$ has a left adjoint, namely $\exists_{Ff}$. Let $\mathscr{S}$ be a pullback square in $\C$, and let $F(\mathscr{S})$ be its image under $F$. The square $F(\mathscr{S})$ is a pullback in $\D$ by assumption, and thus satisfies the Beck--Chevalley condition relative to the full Boolean hyperdoctrine $\R$. Then, by \cref{r:bc-F}, $\mathscr{S}$ satisfies it relative to $\P$, as desired.
\end{proof}

To construct non-full Boolean hyperdoctrines, we consider
$\P = \pws \circ F\op$, where $\pws$ is the powerset hyperdoctrine and $F \colon \C \to \Set$ preserves finite products but not pullbacks.
The preceding proposition ensures that $\P$ is a Boolean hyperdoctrine; the next one shows that it is not full. Before stating the proposition, we make a remark.

\begin{remark}\label{r:mono-quasi}
    Let $f \colon X \to Y$ be a morphism in a category. The following are equivalent.
    \begin{enumerate}
        \item \label{i:mono-for-pb}
        $f$ is a monomorphism.
        \item \label{i:pb-with-id}
        The square $\vcenter{\hbox{\begin{tikzcd}[ampersand replacement=\&, cramped, row sep=1em, column sep=1em, cells={nodes={inner sep=1pt}}]
        X \arrow[r,"\id_X"] \arrow[d,"\id_X"'] \& X \arrow[d,"f"] \\
        X \arrow[r,"f"'] \& Y
        \end{tikzcd}}}$
        is a pullback.
    \end{enumerate}
    Moreover, if the category has pullbacks, these conditions are equivalent also to the following one.
    \begin{enumerate}[resume]
        \item \label{i:quasi-pb-with-id}
        The square $\vcenter{\hbox{\begin{tikzcd}[ampersand replacement=\&, cramped, row sep=1em, column sep=1em, cells={nodes={inner sep=1pt}}]
        X \arrow[r,"\id_X"] \arrow[d,"\id_X"'] \& X \arrow[d,"f"] \\
        X \arrow[r,"f"'] \& Y
        \end{tikzcd}}}$
        is a quasi-pullback.
    \end{enumerate}
    
    Indeed, the equivalence \eqref{i:mono-for-pb}$\Leftrightarrow$\eqref{i:pb-with-id} is well known, the implication \eqref{i:pb-with-id}$\Rightarrow$\eqref{i:quasi-pb-with-id} is trivial, and the implication \eqref{i:quasi-pb-with-id}$\Rightarrow$\eqref{i:pb-with-id} follows from the fact that an epimorphism that is a split monomorphism is an isomorphism. 
\end{remark}

\begin{proposition}[Fullness $\Leftrightarrow$ pullback preservation]\label{p:full-powerset-preserves-pullbacks}
    Let $\C$ be a category with finite products, and $F\colon \C\to\Set$ a functor preserving finite products.
    The composite
    \[
        \C\op \xrightarrow{F\op}\Set\op \xrightarrow{\pws}\BA
    \]
    (which is a Boolean hyperdoctrine by \cref{p:recipe}\eqref{i:recipe-1}) is full if and only if $F$ preserves all pullbacks that exist in $\C$.
\end{proposition}

\begin{proof}
    Let $\P$ denote the composite.
    For every $f\colon X\to Y$ in $\C$, we have
    \[
        \P(f)=(Ff)^{-1}\colon \pws(FY)\to\pws(FX),
    \]
    and its left adjoint is the direct image 
    \begin{equation*}
    (Ff)[-]\colon \pws(FX)\to\pws(FY).
    \end{equation*}
    
    Moreover, the full Beck--Chevalley condition for a pullback square
    \[
    \begin{tikzcd}
        W \arrow[r,"v"] \arrow[d,"u"'] & Y \arrow[d,"g"]\\
        X \arrow[r,"f"'] & Z
        \arrow["\lrcorner"{anchor=center, pos=0.125}, draw=none, from=1-1, to=2-2]
    \end{tikzcd}
    \]
    in $\C$ states that, for every $A\in\pws( FX)$,
    \[
    \P(g)\big(\exists_fA\big)=\exists_v\P(u)(A),
    \]
    which amounts to
    \begin{equation}\label{full-bcc-pullback}
        (Fg)^{-1}\big[(Ff)[A]\big]
        =
        (Fv)\big[(Fu)^{-1}[A]\big].
    \end{equation}

    $(\Rightarrow)$. Suppose that $\P$ is full. Let
    \begin{equation}\label{diag:S}
    \begin{tikzcd}
        W \arrow[r,"v"] \arrow[d,"u"'] & Y \arrow[d,"g"]\\
        X \arrow[r,"f"'] & Z
        \arrow["\lrcorner"{anchor=center, pos=0.125}, draw=none, from=1-1, to=2-2]
    \end{tikzcd}
    \end{equation}
    be a pullback in $\C$. We prove that its image under $F$
    \begin{equation}\label{diag:FS}
    \begin{tikzcd}
        FW \arrow[r,"Fv"] \arrow[d,"Fu"'] & FY \arrow[d,"Fg"]\\
        FX \arrow[r,"Ff"'] & FZ
    \end{tikzcd}
    \end{equation}
    is a pullback in $\Set$, by showing that the canonical map 
    \begin{align*}
        h\colon FW&\longrightarrow FX\times_{FZ}FY\\
        w&\longmapsto ((Fu)(w),(Fv)(w))
    \end{align*}
    is a bijection.
    
    The square \eqref{diag:S} satisfies the Beck--Chevalley condition relative to $\P$ since $\P$ is full. By \cref{r:bc-F}, its image \eqref{diag:FS} under $F$ satisfies the Beck--Chevalley condition relative to $\pws$. By \cref{l:BCC-set}, \eqref{diag:FS} is a quasi-pullback, and thus $h$ is surjective.

    It remains to prove that $h$ is injective. Since the original square is a pullback, the morphism
    \[
        m\coloneqq\langle u,v\rangle\colon W\to X\times Y
    \]
    is a monomorphism. By \cref{r:mono-quasi}, the square
    \begin{equation}\label{diag:m-inj}
    \begin{tikzcd}
        W \arrow[r,"\id_W"] \arrow[d,"\id_W"'] & W \arrow[d,"m"]\\
        W \arrow[r,"m"'] & X\times Y
        \arrow["\lrcorner"{anchor=center, pos=0.125}, draw=none, from=1-1, to=2-2]
    \end{tikzcd}
    \end{equation}
    is a pullback. Since $\P$ is full, the square \eqref{diag:m-inj} satisfies the Beck--Chevalley condition. By \cref{r:bc-F}, its image 
    \[\begin{tikzcd}
        FW \arrow[r,"\id_{FW}"] \arrow[d,"\id_{FW}"'] & FW \arrow[d,"Fm"]\\
        FW \arrow[r,"Fm"'] & F(X\times Y)
    \end{tikzcd}\]
    under $F$ satisfies the Beck--Chevalley condition relative to $\pws$, and so, by \cref{l:BCC-set}, it is a quasi-pullback. Thus, by \cref{r:mono-quasi}, $Fm$ is injective.
    
    In particular, since $F$ preserves finite products and $m$ is the product map induced by $u \colon W \to X$ and $v \colon W \to Y$, the map $Fm$ is $\ple{ Fu,Fv}\colon FW\to FX\times FY$, and thus $\ple{ Fu,Fv}$ is injective.
    
    Since $h$ is a restriction (on the codomain) of $\ple{ Fu,Fv}$, $h$ is injective as well.
    Thus, $h$ is bijective, and hence \eqref{diag:FS} is a pullback in $\Set$, as desired.
    
    $(\Leftarrow)$. This follows from \cref{p:recipe}\eqref{i:recipe-2} (since $\pws$ is full, see \cref{ex:subset-hyp}).
\end{proof}

To get a non-full Boolean hyperdoctrine, we are now left to find a functor $F \colon \C\to \Set$, with $\C$ a category with finite products, that preserves finite products but not pullbacks: then the composite
\[
    \C\op \xrightarrow{F\op}\Set\op \xrightarrow{\pws}\BA
\]
is a non-full Boolean hyperdoctrine.
In the following subsections, we exhibit some examples.

\subsection{A simple example of a non-full Boolean hyperdoctrine}

A functor $\C \to \Set$ that looks like a ``connected component functor'' is likely to preserve finite products but not pullbacks.
While we may consider the usual connected component functor on topological spaces, we opt for a simpler one, which does not need the topological structure.

\begin{definition}[The non-emptiness functor]
    We let 
    \[
    N\colon\Set\to\Set
    \]
    denote the \emph{non-emptiness functor}, i.e., the functor defined on objects as
    \begin{align*}
        N\colon\Set&\longrightarrow\Set\\
        X&\longmapsto\begin{cases}\varnothing&\text{if $X=\varnothing$,}\\
        \{*\}&\text{if $X\neq\varnothing$},
        \end{cases}
    \end{align*}
    and which maps a morphism $f\colon X\to Y$ in $\Set$ to the unique function $N(X)\to N(Y)$.
\end{definition}

The functor $N\colon\Set\to\Set$ can be thought of as a ``connected component''-like functor, by thinking of a set as equipped with the indiscrete topology.

\begin{lemma} \label{l:nonemptiness-preserves-prod-not-pull}
    The non-emptiness functor $N \colon \Set \to \Set$ preserves finite products but not pullbacks.
\end{lemma}

\begin{proof}
    It is easily seen that the non-emptiness functor $N\colon\Set\to\Set$ preserves finite products.
    To show that $N$ does not preserve pullbacks, consider the following commutative square in $\Set$:
    \[
        \begin{tikzcd}
            \varnothing \arrow[r,"e"] \arrow[d,"e'"'] &
            \{b\} \arrow[d,"i'"] \\
            \{a\} \arrow[r,"i"'] &
            \{a,b\}.
        \end{tikzcd}
    \]
    Here $e,e'$ are the unique maps, and $i,i'$ are the inclusions of the sets. Since $\{a\}\cap\{b\}=\varnothing$, this square is a pullback in $\Set$.
    Applying $N$ gives the following square, which is not a pullback in $\Set$.
    \[
        \begin{tikzcd}[baseline=(\tikzcdmatrixname-2-1.base)]
            \varnothing \arrow[r] \arrow[d] &
            \{*\} \arrow[d] \\
            \{*\} \arrow[r] &
            \{*\}
        \end{tikzcd}\qedhere
    \]
\end{proof}

\begin{theorem}[Boolean hyperdoctrine $\not\Rightarrow$ full]\label{t:BC-pi0-counterexample}
    The \emph{indiscrete functor}
    \begin{align*}
        \I\colon\Set\op&\longrightarrow\BA\\
        X&\longmapsto\begin{cases}1&\text{if $X=\varnothing$,}\\
        2&\text{if $X\neq\varnothing$},
        \end{cases}
    \end{align*}
    which maps a morphism $f\colon X\to Y$ in $\Set$ to the unique Boolean homomorphism $\I(Y)\to \I(X)$, is a non-full Boolean hyperdoctrine.
\end{theorem}

\begin{proof}
    Let us observe that $\I \colon \Set\op \to \BA$ is naturally isomorphic to the composite of the following two functors
    \[
        \Set\op \xrightarrow{N\op} \Set\op\xrightarrow{\pws}\BA.
    \]
    
    By \cref{l:nonemptiness-preserves-prod-not-pull}, the non-emptiness functor $N \colon \Set \to \Set$ preserves finite products but not pullbacks.
    Therefore, $\I$ is a Boolean hyperdoctrine by \cref{p:recipe}\eqref{i:recipe-1}, and is not full by \cref{p:full-powerset-preserves-pullbacks}.
\end{proof}

The name \emph{indiscrete functor} should suggest that $\I$ is naturally isomorphic to the functor that maps $X$ to the Boolean algebra $\{\varnothing, X\}$ (which has two elements if $X \neq \varnothing$ and one otherwise), which can be seen as the set of opens of $X$ when equipped with the indiscrete topology.

Thus, ordinary Beck--Chevalley does not imply full Beck--Chevalley for Boolean hyperdoctrines in general (even when the base category is complete and cocomplete).

\subsection{Examples of non-full Boolean hyperdoctrines over categories of contexts}
\label{ss:contexts-examples}
The base category of the non-full Boolean hyperdoctrine of
\cref{t:BC-pi0-counterexample} is $\Set$.
We now exhibit a non-full Boolean hyperdoctrine over the category of contexts of a first-order language.
Let us observe that a pullback-preserving functor preserves monomorphisms (this follows from \cref{r:mono-quasi}).

Let $\F$ be a one-sorted first-order functional language, and let $\Ctx_\F$ be the category of contexts of \cref{ex:synt-doc}.
Every $n$-ary function symbol $\mathsf f\in\F$
induces a monomorphism in $\Ctx_\F$
\begin{align*}
    \widehat {\mathsf f}\colon \{x_1,\dots,x_n\}&\longrightarrow\{y\},
    \\
    y&\longmapsto\mathsf f(x_1,\dots,x_n).
\end{align*}
Indeed, for morphisms
$\sigma,\tau\colon Z\to\{x_1,\dots,x_n\}$, the equality
$\widehat {\mathsf f}\circ\sigma=\widehat {\mathsf f}\circ\tau$ means that
\[
    \mathsf f(\sigma(x_1),\dots,\sigma(x_n))
    =
    \mathsf f(\tau(x_1),\dots,\tau(x_n))
\]
as syntactic terms.
By the inductive definition of terms, this implies $\sigma(x_i)=\tau(x_i)$ for all $i$, and so $\sigma=\tau$.

Every $\F$-structure $M$ determines an interpretation functor
\begin{align*}
    F_M\colon\Ctx_\F&\longrightarrow\Set\\
    Y&\longmapsto M^Y,
\end{align*}
which maps a morphism $\sigma\colon Y\to Y'$ to the function $M^Y\to M^{Y'}$ whose $y'$-th component is the interpretation of the term $\sigma(y')$ in $M$. 
This functor preserves finite products, since products in $\Ctx_\F$ are disjoint unions of contexts.
For every $n$-ary function symbol $\mathsf f\in\F$ we have $F_M(\widehat {\mathsf f})=\mathsf f^M\colon M^n\to M$.
Thus, if some $\mathsf f^M$ is not injective, the preceding observation shows that $F_M$ does not preserve all existing pullbacks; then, by \cref{p:recipe}\eqref{i:recipe-1} and \ref{p:full-powerset-preserves-pullbacks}, the composite
\[
    \P_M\coloneqq
    \bigl(
        \Ctx_\F\op\xrightarrow{F_M\op}
        \Set\op\xrightarrow{\pws}\BA
    \bigr)
\]
is a non-full Boolean hyperdoctrine.

Consequently, any one-sorted functional language containing a function symbol $\mathsf{f}$ of arity $\geq 1$ gives rise to such an example: take a structure with at least two elements and interpret $\mathsf{f}$ as a constant function.
The fact that the arity of $\mathsf{f}$ is $\geq 1$ ensures that this interpretation is not injective. The minimal example is carried out in detail in the following.

\begin{example}[An example over a category of contexts]\label{ex:non-full-contexts}
    Let $\F$ consist of a single unary function symbol $\mathsf f$,
    and let $M$ be the $\F$-structure with underlying set
    $\{a_1,a_2\}$, where $a_1\neq a_2$, and interpretation
    \[
        \mathsf f^M(a_1)=\mathsf f^M(a_2)=a_1.
    \]
    By the preceding argument, the composite
    \[
        \P_M=\pws\circ F_M\op\colon\Ctx_\F\op\longrightarrow\BA
    \]
    is a non-full Boolean hyperdoctrine. We show that this hyperdoctrine is isomorphic to the syntactic hyperdoctrine of the first-order theory $\T$ in the language $\F$ axiomatized by the sentence
    \[
        \exists x_1\exists x_2\,\bigl(
            x_1\neq x_2
            \land \forall x\,(x=x_1\lor x=x_2)
            \land \forall x\,(\mathsf f(x)=x_1)
        \bigr),
    \]
    which says that there are exactly two elements and that
    $\mathsf f$ is constant.
    Up to isomorphism, $\T$ has a unique model, namely the
    structure $M$ defined above. In particular, $\T$ is consistent and complete, so that, for all formulas $\varphi,\psi$ in a context $Y$,
    \[
        \T\vdash\varphi\Rightarrow_Y\psi
        \quad\text{if and only if}\quad
        \varphi^M\subseteq\psi^M.
    \]
    Hence, for every context $Y$, interpretation in $M$ gives an order-reflecting (hence injective) Boolean homomorphism
    \begin{align*}
        \LT^\T(Y)&\longrightarrow\pws(F_M(Y))\\
        [\varphi]&\longmapsto\varphi^M.
    \end{align*}
    It is also surjective: the formulas $\mathsf f(x)=x$ and $\lnot(\mathsf f(x)=x)$ define
    $a_1$ and $a_2$ respectively, and so every singleton $\{\vec a\}\in\pws(M^Y)$ is defined by a conjunction of such formulas; since $M^Y$ is finite, every subset of $M^Y$ is then defined by a finite disjunction of such conjunctions.
    Since interpretation commutes with substitution, these homomorphisms are natural in $Y$, and hence yield an isomorphism
    \[
        \LT^\T\cong\pws\circ F_M\op.
    \]
\end{example}

Thus, the syntactic hyperdoctrine of a first-order theory
need not be full.
In the next section we prove that the syntactic hyperdoctrine of every first-order theory in a many-sorted purely relational language with equality is full (\cref{c:syn-full}).
We first establish the underlying categorical result over
$\FinSet\op$ and then extend it to finite sets equipped with sorts.

\section{Fullness for relational languages}
\label{s:bool-doct-finset}
We now show that every Boolean hyperdoctrine
over $\FinSet\op$ is full. This means (see, e.g., \cite[Cor.~3.10]{AbbadiniGuffanti}) that every syntactic hyperdoctrine of a theory in a one-sorted relational language with equality is full. 
We then extend the result to the many-sorted setting in Subsection~\ref{ss:many-sorted}, obtaining fullness for the syntactic hyperdoctrines of theories in many-sorted purely relational languages with equality.

The key is that every morphism in $\FinSet$ decomposes into a surjective function followed by an injective function, and that every surjective function decomposes into a finite sequence of maps, each of which collapses at most a pair of points.
This will allow us to decompose an arbitrary pushout square in $\FinSet$ into a pasting of pushout squares of a simple form, for which we know the Beck--Chevalley condition to hold.

In the one-sorted case, there is also an indirect route to \cref{t:all-full-finset} through spaces of types. Indeed, every Boolean hyperdoctrine over $\FinSet\op$ is isomorphic to the Lindenbaum--Tarski hyperdoctrine of a one-sorted theory $\T$ in a purely relational language with equality (cf., e.g.,~\cite[Cor.~3.10]{AbbadiniGuffanti}). By G\"odel's completeness theorem, the fiber-wise Stone dual of this hyperdoctrine is naturally isomorphic to the functor of complete types of $\T$ (see, e.g., \cite[Sec.~2.2]{GehrkeJaklEtAL2023}). Therefore, it  suffices to verify the Beck--Chevalley condition on the dual side, on functors of complete types. Kamsma proved the corresponding statement for coherent types: every pushout in $\FinSet$ induces a square of type spaces satisfying the (fiber-wise dual of the) Beck--Chevalley condition \cite[Prop.~4.11]{Kamsma2023}.

\subsection{Shapes of pushouts in \texorpdfstring{$\FinSet$}{FinSet}}

\begin{definition}[Pair-identifying map]
    A \emph{pair-identifying map} is a function $X \to Y$ for which there is $y_0 \in Y$ such that the preimage of $\{y_0\}$ has cardinality $2$, and, for all $y \in Y \setminus \{y_0\}$, the preimage of $\{y\}$ has cardinality $1$.
    
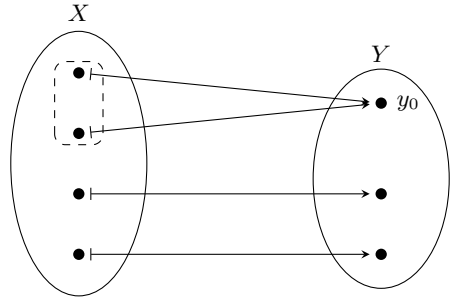
\[
   \begin{tikzpicture}[
        scale=0.8,
        dot/.style={circle, fill, inner sep=1.5pt},
        every node/.style={font=\small},
        >=stealth
        ]
        \draw (0,0.8) ellipse (1.2 and 1.75);
        \draw (4,0.6) ellipse (1.2 and 1.45);

        \node at (0,2.8) {$X$};
        \node at (4,2.25) {$Y$};

        \node[dot] (x1) at (0,2) {};
        \node[dot] (x2) at (0,1.2) {};
        \node[dot] (x3) at (0,0.4) {};
        \node[dot] (x4) at (0,-0.4) {};

        \node[dot,label=right:{$y_0$}] (y) at (4,1.6) {};
        \node[dot] (y1) at (4,0.4) {};
        \node[dot] (y2) at (4,-0.4) {};

        \draw[dashed, rounded corners]
            (-0.32,1.05) rectangle (0.32,2.15);

        \draw[|->, shorten >=2pt, shorten <=2pt] (x1) -- (y);
        \draw[|->, shorten >=2pt, shorten <=2pt] (x2) -- (y);
        \draw[|->, shorten >=2pt, shorten <=2pt] (x3) -- (y1);
        \draw[|->, shorten >=2pt, shorten <=2pt] (x4) -- (y2);
    \end{tikzpicture}
    \]
\end{definition}

\begin{remark}\label{r:decomposition_surj}
    Every surjective function between finite sets can be decomposed into a finite sequence of pair-identifying maps and isomorphisms.
\end{remark}

\begin{remark}[Pushouts of pair-identifying maps]\label{r:pair-id}
    The pushout of a span of two pair-identifying maps
    \[
    \begin{tikzcd}
        \bullet \arrow[two heads]{r}{f} \arrow[two heads, swap]{d}{g} & \bullet\\
        \bullet
    \end{tikzcd}
    \]
    has three possible shapes, depending on how the two pairs of identified elements intersect, as we next illustrate.
    Let us denote by $\{a_1,a_2\}$ the pair identified by $f$ and by $\{b_1,b_2\}$ the pair identified by $g$. 
    In each case below, we display an intuitive picture of the pushout on the left, and its formal description as a diagram in $\FinSet$ on the right. 

    \begin{enumerate}[label=(\roman*), ref=\roman*]
        \item\label{eq:pi-1} If $\{a_1,a_2\}\cap\{b_1,b_2\}=\varnothing$:  the pushout in $\FinSet$ has the following shape.

        \smallskip
        
        \noindent
        \begin{minipage}{0.5\linewidth}
            
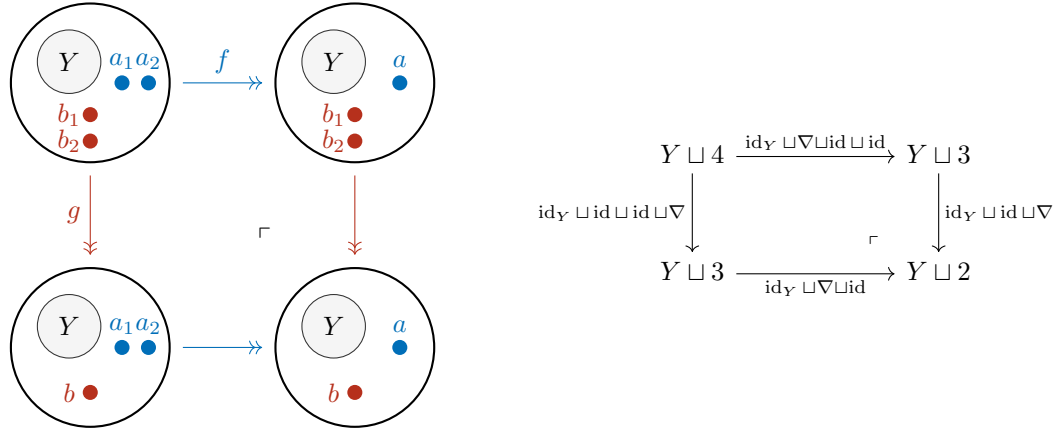
\centering
            \begin{tikzpicture}[scale=0.7]
                
                \tikzset{
                  bigoval/.style={thick},
                  smalloval/.style={draw=black!80, fill=black!4},
                  adot/.style={circle, fill=NavyBlue, inner sep=2pt},
                  bdot/.style={circle, fill=BrickRed, inner sep=2pt},
                  alabel/.style={font=\itshape, text=NavyBlue},
                  blabel/.style={font=\itshape, text=BrickRed}
                }
                
                \draw[bigoval] (0,0) ellipse (1.5cm and 1.5cm);
                \draw[smalloval] (-0.4,0.4) ellipse (0.6cm and 0.6cm);
                \node[font=\itshape] at (-0.4,0.4) {$Y$};
                
                \node[adot] at (0.6,0) {};
                \node[alabel] at (0.6,0.4) {$a_1$};
                \node[adot] at (1.1,0) {};
                \node[alabel] at (1.1,0.4) {$a_2$};
                
                \node[bdot] at (0,-0.6) {};
                \node[blabel] at (-0.4,-0.6) {$b_1$};
                \node[bdot] at (0,-1.1) {};
                \node[blabel] at (-0.4,-1.1) {$b_2$};
                
                \draw[bigoval] (5,0) ellipse (1.5cm and 1.5cm);
                \draw[smalloval] (4.6,0.4) ellipse (0.6cm and 0.6cm);
                \node[font=\itshape] at (4.6,0.4) {$Y$};
                
                \node[adot] at (5.85,0) {};
                \node[alabel] at (5.85,0.4) {$a$};
                
                \node[bdot] at (5,-0.6) {};
                \node[blabel] at (4.6,-0.6) {$b_1$};
                \node[bdot] at (5,-1.1) {};
                \node[blabel] at (4.6,-1.1) {$b_2$};
                
                \draw[bigoval] (0,-5) ellipse (1.5cm and 1.5cm);
                \draw[smalloval] (-0.4,-4.6) ellipse (0.6cm and 0.6cm);
                \node[font=\itshape] at (-0.4,-4.6) {$Y$};
                
                \node[adot] at (0.6,-5) {};
                \node[alabel] at (0.6,-4.6) {$a_1$};
                \node[adot] at (1.1,-5) {};
                \node[alabel] at (1.1,-4.6) {$a_2$};
                
                \node[bdot] at (0,-5.85) {};
                \node[blabel] at (-0.4,-5.85) {$b$};
                
                \draw[bigoval] (5,-5) ellipse (1.5cm and 1.5cm);
                \draw[smalloval] (4.6,-4.6) ellipse (0.6cm and 0.6cm);
                \node[font=\itshape] at (4.6,-4.6) {$Y$};
                
                \node[adot] at (5.85,-5) {};
                \node[alabel] at (5.85,-4.6) {$a$};
                
                \node[bdot] at (5,-5.85) {};
                \node[blabel] at (4.6,-5.85) {$b$};
                
                \draw[->>, draw=NavyBlue] (1.75,0) -- (3.25,0)
                    node[midway,above,text=NavyBlue] {$f$};
                
                \draw[->>, draw=BrickRed] (0,-1.75) -- (0,-3.25)
                    node[midway,left,text=BrickRed] {$g$};
                
                \draw[->>, draw=BrickRed] (5,-1.75) -- (5,-3.25);
                
                \draw[->>, draw=NavyBlue] (1.75,-5) -- (3.25,-5);
                
                \node at (3.3,-2.9) {$\ulcorner$};

            \end{tikzpicture}

        \end{minipage}%
        \begin{minipage}{0.5\linewidth}
            \centering
    
            \begin{tikzcd}[column sep = 6em, row sep = 3em]
             	{Y \sqcup 4} & {Y\sqcup 3} \\
             	{Y \sqcup 3} & {Y\sqcup 2}
             	\arrow["{\id_Y\sqcup \nabla \sqcup \id\sqcup \id}", from=1-1, to=1-2]
             	\arrow["{\id_Y\sqcup \id\sqcup\id\sqcup \nabla}"', from=1-1, to=2-1]
             	\arrow["\ulcorner"{anchor=center, pos=0.875}, draw=none, from=1-1, to=2-2]
             	\arrow["{\id_Y\sqcup \nabla \sqcup \id}"', from=2-1, to=2-2]
             	\arrow["{\id_Y\sqcup \id\sqcup \nabla}", from=1-2, to=2-2]
            \end{tikzcd}

            \ \newline
            \ \newline
            \small (Here, $\nabla \colon 1 \sqcup 1 \to 1$ denotes the codiagonal.)
        \end{minipage}
        
        \smallskip
        
        \item\label{eq:pi-2} If $|\{a_1,a_2\}\cap\{b_1,b_2\}|=1$: assuming, without loss of generality, that the intersection consists of $a_2 = b_1$, the pushout in $\FinSet$ has the following shape.
    
        \smallskip
         
        \noindent
        \begin{minipage}{0.5\linewidth}
            
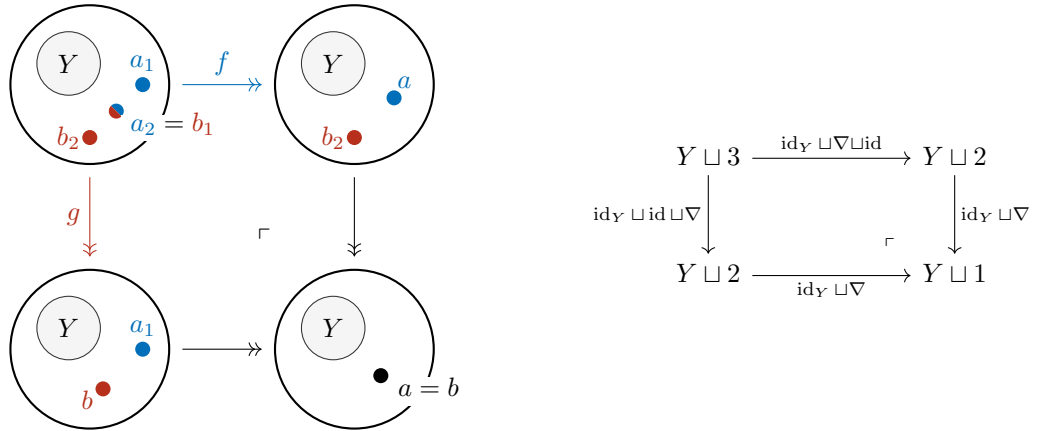
\centering
            \begin{tikzpicture}[scale=0.7]
                
                \tikzset{
                  bigoval/.style={thick},
                  smalloval/.style={draw=black!80, fill=black!4},
                  adot/.style={circle, fill=NavyBlue, inner sep=2pt},
                  bdot/.style={circle, fill=BrickRed, inner sep=2pt},
                  kdot/.style={circle, fill=black, inner sep=2pt},
                  alabel/.style={font=\itshape, text=NavyBlue},
                  blabel/.style={font=\itshape, text=BrickRed},
                  abdot/.pic={
                    \begin{scope}
                      \clip (0,0) circle[radius=2.83pt];
                      \fill[NavyBlue] (-3pt,3pt) -- (3pt,3pt) -- (3pt,-3pt) -- cycle;
                      \fill[BrickRed] (-3pt,3pt) -- (-3pt,-3pt) -- (3pt,-3pt) -- cycle;
                    \end{scope}
                  }
                }
                
                \draw[bigoval] (0,0) ellipse (1.5cm and 1.5cm);
                \draw[smalloval] (-0.4,0.4) ellipse (0.6cm and 0.6cm);
                \node[font=\itshape] at (-0.4,0.4) {$Y$};
                
                \node[adot] at (1,0) {};
                \node[alabel] at (1,0.4) {$a_1$};
                
                \pic at (0.5,-0.5) {abdot};
                \node[font=\itshape, fill=white, text height=0.6ex, text depth=0.07ex] at (1.57,-0.82) 
                    {${\color{NavyBlue}a_2}={\color{BrickRed}b_1}$};
                
                \node[bdot] at (0,-1) {};
                \node[blabel] at (-0.4,-1) {$b_2$};
                
                \draw[bigoval] (5,0) ellipse (1.5cm and 1.5cm);
                \draw[smalloval] (4.6,0.4) ellipse (0.6cm and 0.6cm);
                \node[font=\itshape] at (4.6,0.4) {$Y$};
                
                \node[adot] at (5.75,-0.25) {};
                \node[alabel] at (5.95,0.05) {$a$};
                
                \node[bdot] at (5,-1) {};
                \node[blabel] at (4.6,-1) {$b_2$};
                
                \draw[bigoval] (0,-5) ellipse (1.5cm and 1.5cm);
                \draw[smalloval] (-0.4,-4.6) ellipse (0.6cm and 0.6cm);
                \node[font=\itshape] at (-0.4,-4.6) {$Y$};
                
                \node[adot] at (1,-5) {};
                \node[alabel] at (1,-4.6) {$a_1$};
                
                \node[bdot] at (0.25,-5.75) {};
                \node[blabel] at (-0.05,-5.95) {$b$};
                
                \draw[bigoval] (5,-5) ellipse (1.5cm and 1.5cm);
                \draw[smalloval] (4.6,-4.6) ellipse (0.6cm and 0.6cm);
                \node[font=\itshape] at (4.6,-4.6) {$Y$};
                
                \node[kdot] at (5.5,-5.5) {};
                \node[font=\itshape, fill=white, text height=0.9ex, text depth=0.06ex] at (6.4,-5.8) {$a=b$};
                
                \draw[->>, draw=NavyBlue] (1.75,0) -- (3.25,0)
                    node[midway,above,text=NavyBlue] {$f$};
                
                \draw[->>, draw=BrickRed] (0,-1.75) -- (0,-3.25)
                    node[midway,left,text=BrickRed] {$g$};
                
                \draw[->>] (5,-1.75) -- (5,-3.25);
                
                \draw[->>] (1.75,-5) -- (3.25,-5);
                
                \node at (3.3,-2.9) {$\ulcorner$};
                
            \end{tikzpicture}
        \end{minipage}%
        \begin{minipage}{0.5\linewidth}
            \centering
            \begin{tikzcd}[column sep = 6em, row sep = 3em]
            	{Y \sqcup 3} & {Y\sqcup 2} \\
            	{Y \sqcup 2} & {Y\sqcup 1}
            	\arrow["{\id_Y\sqcup  \nabla\sqcup\id}", from=1-1, to=1-2]
            	\arrow["{\id_Y\sqcup \id\sqcup \nabla}"', from=1-1, to=2-1]
            	\arrow["\ulcorner"{anchor=center, pos=0.875}, draw=none, from=1-1, to=2-2]
            	\arrow["{\id_Y\sqcup \nabla}", from=1-2, to=2-2]
            	\arrow["{\id_Y\sqcup \nabla}"', from=2-1, to=2-2]
            \end{tikzcd}
        \end{minipage}
        
        \smallskip
        
        \item\label{eq:pi-3} 
        If $\{a_1,a_2\}=\{b_1,b_2\}$: the pushout in $\FinSet$ has the following shape.
        
        \smallskip
        
        \noindent
        \begin{minipage}{0.5\linewidth}
            
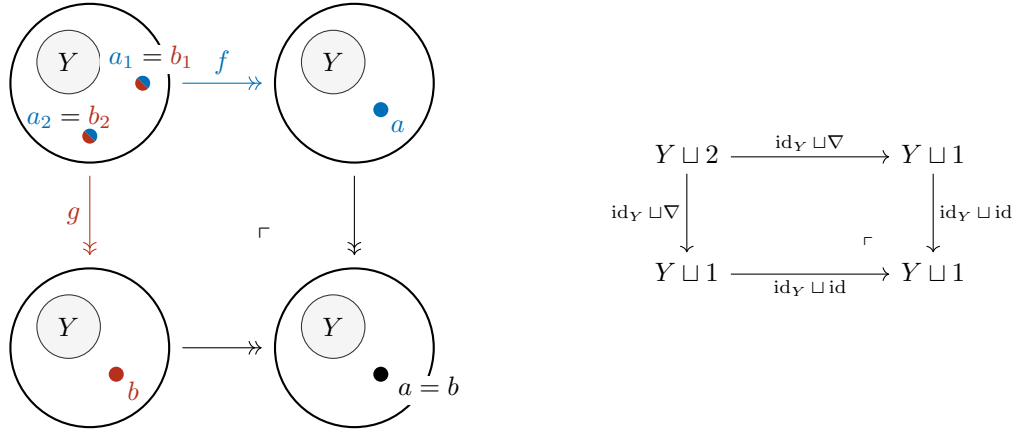
\centering
            \begin{tikzpicture}[scale=0.7]
                
                \tikzset{
                  bigoval/.style={thick},
                  smalloval/.style={draw=black!80, fill=black!4},
                  adot/.style={circle, fill=NavyBlue, inner sep=2pt},
                  bdot/.style={circle, fill=BrickRed, inner sep=2pt},
                  kdot/.style={circle, fill=black, inner sep=2pt},
                  alabel/.style={font=\itshape, text=NavyBlue},
                  blabel/.style={font=\itshape, text=BrickRed},
                  klabel/.style={font=\itshape},
                  abdot/.pic={
                    \begin{scope}
                      \clip (0,0) circle[radius=2.83pt];
                      \fill[NavyBlue] (-3pt,3pt) -- (3pt,3pt) -- (3pt,-3pt) -- cycle;
                      \fill[BrickRed] (-3pt,3pt) -- (-3pt,-3pt) -- (3pt,-3pt) -- cycle;
                    \end{scope}
                  }
                }
                
                \draw[bigoval] (0,0) ellipse (1.5cm and 1.5cm);
                \draw[smalloval] (-0.4,0.4) ellipse (0.6cm and 0.6cm);
                \node[font=\itshape] at (-0.4,0.4) {$Y$};
                
                \pic at (1,0) {abdot};
                \node[font=\itshape, fill=white, text height=0.9ex, text depth=0.06ex] at (1.16,0.45)
                    {${\color{NavyBlue}a_1}={\color{BrickRed}b_1}$};
                
                \pic at (0,-1) {abdot};
                \node[font=\itshape] at (-0.39,-0.64)
                    {${\color{NavyBlue}a_2}={\color{BrickRed}b_2}$};
                
                \draw[bigoval] (5,0) ellipse (1.5cm and 1.5cm);
                \draw[smalloval] (4.6,0.4) ellipse (0.6cm and 0.6cm);
                \node[font=\itshape] at (4.6,0.4) {$Y$};
                
                \node[adot] at (5.5,-0.5) {};
                \node[alabel] at (5.82,-0.82) {$a$};
                
                \draw[bigoval] (0,-5) ellipse (1.5cm and 1.5cm);
                \draw[smalloval] (-0.4,-4.6) ellipse (0.6cm and 0.6cm);
                \node[font=\itshape] at (-0.4,-4.6) {$Y$};
                
                \node[bdot] at (0.5,-5.5) {};
                \node[blabel] at (0.82,-5.82) {$b$};
                
                \draw[bigoval] (5,-5) ellipse (1.5cm and 1.5cm);
                \draw[smalloval] (4.6,-4.6) ellipse (0.6cm and 0.6cm);
                \node[font=\itshape] at (4.6,-4.6) {$Y$};
                
                \node[kdot] at (5.5,-5.5) {};
                \node[font=\itshape, fill=white, text height=0.9ex, text depth=0.06ex] at (6.4,-5.8) {$a=b$};
                
                \draw[->>, draw=NavyBlue] (1.75,0) -- (3.25,0)
                    node[midway,above,text=NavyBlue] {$f$};
                
                \draw[->>, draw=BrickRed] (0,-1.75) -- (0,-3.25)
                    node[midway,left,text=BrickRed] {$g$};
                
                \draw[->>] (5,-1.75) -- (5,-3.25);
                
                \draw[->>] (1.75,-5) -- (3.25,-5);
                
                \node at (3.3,-2.9) {$\ulcorner$};
                
            \end{tikzpicture}
        \end{minipage}%
        \noindent
        \begin{minipage}{0.5\linewidth}
            \centering
            \begin{tikzcd}[column sep = 6em, row sep=3em]
            	{Y \sqcup 2} & {Y\sqcup 1} \\
            	{Y\sqcup 1} & {Y\sqcup 1}
            	\arrow["{\id_Y\sqcup \nabla}", from=1-1, to=1-2]
            	\arrow["{\id_Y\sqcup \nabla}"' , from=1-1, to=2-1]
            	\arrow["\ulcorner"{anchor=center, pos=0.875}, draw=none, from=1-1, to=2-2]
            	\arrow["{\id_Y\sqcup\id}", from=1-2, to=2-2]
            	\arrow["{\id_Y\sqcup\id}"', from=2-1, to=2-2]
            \end{tikzcd}
        \end{minipage}
        
        \smallskip
        
    \end{enumerate}
    Let us observe that the two morphisms in the cospans $\inlinecorner$ in these pushouts are either pair-identifying maps (in \eqref{eq:pi-1} and \eqref{eq:pi-2}) or isomorphisms (in \eqref{eq:pi-3}). 
\end{remark}

\begin{lemma}[Grid decomposition for pushouts]\label{l:surj}
    The pushout square in $\FinSet$ of a span of two surjective functions
    \[
    \begin{tikzcd}
        \bullet \arrow[two heads]{r}{f} \arrow[two heads, swap]{d}{g} & \bullet\\
        \bullet
    \end{tikzcd}
    \]
    can be decomposed into a grid of pushout squares where every involved morphism is a pair-identifying map or an isomorphism.
    \[
    \begin{tikzcd}[column sep = 1.3em, row sep=1.3em]
    	\bullet & \bullet & \bullet & \bullet & \bullet & \bullet \\
    	\bullet & \bullet & \bullet & \bullet & \bullet & \bullet \\
    	\bullet & \bullet & \bullet & \bullet & \bullet & \bullet \\
    	\bullet & \bullet & \bullet & \bullet & \bullet & \bullet \\
    	\bullet & \bullet & \bullet & \bullet & \bullet & \bullet
    	\arrow[two heads, from=1-1, to=1-2]
    	\arrow["f"{description}, curve={height=-24pt}, from=1-1, to=1-6]
    	\arrow[two heads, from=1-1, to=2-1]
    	\arrow["g"{description}, curve={height=24pt}, from=1-1, to=5-1]
    	\arrow[two heads, from=1-2, to=1-3]
    	\arrow[two heads, from=1-2, to=2-2]
    	\arrow[two heads, from=1-3, to=1-4]
    	\arrow[two heads, from=1-3, to=2-3]
    	\arrow["\dots"{marking, allow upside down}, draw=none, from=1-4, to=1-5]
    	\arrow[two heads, from=1-4, to=2-4]
    	\arrow[two heads, from=1-5, to=1-6]
    	\arrow[two heads, from=1-5, to=2-5]
    	\arrow[two heads, from=1-6, to=2-6]
    	\arrow[two heads, from=2-1, to=2-2]
    	\arrow[two heads, from=2-1, to=3-1]
    	\arrow[two heads, from=2-2, to=2-3]
    	\arrow[two heads, from=2-2, to=3-2]
    	\arrow[two heads, from=2-3, to=2-4]
    	\arrow[two heads, from=2-3, to=3-3]
    	\arrow["\dots"{marking, allow upside down}, draw=none, from=2-4, to=2-5]
    	\arrow[two heads, from=2-4, to=3-4]
    	\arrow[two heads, from=2-5, to=2-6]
    	\arrow[two heads, from=2-5, to=3-5]
    	\arrow[two heads, from=2-6, to=3-6]
    	\arrow[two heads, from=3-1, to=3-2]
    	\arrow["\vdots"{description}, draw=none, from=3-1, to=4-1]
    	\arrow[two heads, from=3-2, to=3-3]
    	\arrow["\vdots"{description}, draw=none, from=3-2, to=4-2]
    	\arrow[two heads, from=3-3, to=3-4]
    	\arrow["\vdots"{description}, draw=none, from=3-3, to=4-3]
    	\arrow["\dots"{marking, allow upside down}, draw=none, from=3-4, to=3-5]
    	\arrow["\vdots"{description}, draw=none, from=3-4, to=4-4]
    	\arrow[two heads, from=3-5, to=3-6]
    	\arrow["\vdots"{description}, draw=none, from=3-5, to=4-5]
    	\arrow["\vdots"{description}, draw=none, from=3-6, to=4-6]
    	\arrow[two heads, from=4-1, to=4-2]
    	\arrow[two heads, from=4-1, to=5-1]
    	\arrow[two heads, from=4-2, to=4-3]
    	\arrow[two heads, from=4-2, to=5-2]
    	\arrow[two heads, from=4-3, to=4-4]
    	\arrow[two heads, from=4-3, to=5-3]
    	\arrow["\dots"{marking, allow upside down}, draw=none, from=4-4, to=4-5]
    	\arrow[two heads, from=4-4, to=5-4]
    	\arrow[two heads, from=4-5, to=4-6]
    	\arrow[two heads, from=4-5, to=5-5]
    	\arrow[two heads, from=4-6, to=5-6]
    	\arrow[two heads, from=5-1, to=5-2]
    	\arrow[two heads, from=5-2, to=5-3]
    	\arrow[two heads, from=5-3, to=5-4]
    	\arrow["\dots"{marking, allow upside down}, draw=none, from=5-4, to=5-5]
    	\arrow[two heads, from=5-5, to=5-6]
    \end{tikzcd}
    \]
\end{lemma}

\begin{proof}
    By \cref{r:decomposition_surj}, both $f$ and $g$ can be decomposed into a finite sequence of pair-identifying maps and isomorphisms. By \cref{r:pair-id} in the pushout of pair-identifying maps we have isomorphisms or pair-identifying maps. Moreover, the pushout of an isomorphism along a pair-identifying map is trivially given by an isomorphism and a pair-identifying map, and in the pushout of two isomorphisms we have isomorphisms.
    Applying these observations recursively on the pushout squares from the upper-left corner to the lower-right corner, we get the result.
\end{proof}

\subsection{Some squares with the Beck--Chevalley condition}

In this subsection, for an arbitrary Boolean hyperdoctrine, we exhibit some squares with the Beck--Chevalley condition.
In particular, we will deal with the pullback squares in $\FinSet\op$ in \cref{r:pair-id}.
We will use these facts in \cref{ss:bool-hyp-full} to show that every Boolean hyperdoctrine over $\FinSet\op$ is full.

\begin{remark}\label{r:univ}
    Let $\P$ be a Boolean hyperdoctrine over $\C$, and $f\colon X\to Y$ a morphism in $\C$.
    We write $\forall_f\coloneqq\lnot\exists_f\lnot\colon \P(X)\to\P(Y)$. Clearly, $\forall_f$ is the right adjoint to $\P(f)$. Moreover, if $\P$ satisfies the Beck--Chevalley condition for the existential quantifier for a given commutative square, then $\P$ satisfies the Beck--Chevalley condition for the universal quantifier for the same square.
\end{remark}

The following proposition shows that the Beck--Chevalley condition does not depend on the chosen orientation of the pullback square.

\begin{proposition}[Symmetry of the Beck--Chevalley Condition]\label{p:inverse-bc}
    Let $\P$ be a Boolean hyperdoctrine over $\C$.
    A commutative square in $\C$
    \begin{equation}\label{diag:bcc-sym}
    \begin{tikzcd}
        W & Y \\
        X & Z
        \arrow["v", from=1-1, to=1-2]
        \arrow["u"', from=1-1, to=2-1]
        \arrow["g", from=1-2, to=2-2]
        \arrow["f"', from=2-1, to=2-2]
    \end{tikzcd}
    \end{equation}
    satisfies the Beck--Chevalley condition if and only if the square
    \begin{equation} \label{diag:bcc-sym-2}
    \begin{tikzcd}
        W & X \\
        Y & Z
        \arrow["u", from=1-1, to=1-2]
        \arrow["v"', from=1-1, to=2-1]
        \arrow["f", from=1-2, to=2-2]
        \arrow["g"', from=2-1, to=2-2]
    \end{tikzcd}
    \end{equation}
    satisfies it. In other words, for every square \eqref{diag:bcc-sym},
    the diagram below on the left commutes if and only if the one on the right does.
    \begin{equation}\label{diag:bcc-sym2}
    \begin{tikzcd}
    	{\P(W)} & {\P(Y)} && {\P(W)} & {\P(Y)} \\
    	{\P(X)} & {\P(Z)} && {\P(X)} & {\P(Z)}
    	\arrow["{\exists_v}", from=1-1, to=1-2]
    	\arrow["{\exists_u}"', from=1-4, to=2-4]
    	\arrow["{\P(v)}"', from=1-5, to=1-4]
    	\arrow["{\exists_g}", from=1-5, to=2-5]
    	\arrow["{\P(u)}", from=2-1, to=1-1]
    	\arrow["{\exists_f}"', from=2-1, to=2-2]
    	\arrow["{\P(g)}"', from=2-2, to=1-2]
    	\arrow["{\P(f)}", from=2-5, to=2-4]
    \end{tikzcd}
    \end{equation}
\end{proposition}

\begin{proof}
    Suppose that the square \eqref{diag:bcc-sym} satisfies the Beck--Chevalley condition, i.e.\ the diagram on the left in \eqref{diag:bcc-sym2}  commutes.
    We shall prove that for every $\alpha\in\P(Y)$, we have
    \[
    \P(f)(\exists_g\alpha)\leq\exists_{u}\P(v)(\alpha).
    \]
    (Indeed, the direction $(\geq)$ holds for every commutative square in $\C$---see \cref{r:leq}.)
    We have:
    \begin{align*}
         &\P(f)(\exists_g\alpha)\leq\exists_{u}\P(v)(\alpha)&&\text{in $\P(X)$}\\
         &\iff (\exists_g\alpha)\leq \forall_f\big(\exists_{u}\P(v)(\alpha)\big)&&\text{in $\P(Z)$}&&\text{(since $\P(f)\dashv\forall_f$, see  Rem.~\ref{r:univ})}\\
         &\iff \alpha\leq\P(g)\forall_f\big(\exists_{u}\P(v)(\alpha)\big)&&\text{in $\P(Y)$}&&\text{(since $\exists_g\dashv\P(g)$)}\\
         &\iff \alpha\leq\forall_v\P(u)\big(\exists_{u}\P(v)(\alpha)\big)&&\text{in $\P(Y)$}&&\text{(by BC for $\forall$, see  Rem.~\ref{r:univ})}
    \end{align*}
    By the units of the adjunctions $\P(v)\dashv\forall_v$ and  $\exists_{u}\dashv\P(u)$ respectively we have
    \[
    \alpha\leq \forall_v\P(v)(\alpha)\leq  \forall_v\P(u)\big(\exists_{u}\P(v)(\alpha)\big),
    \]
    as desired. 
    The converse then follows with the same argument.
\end{proof} 

In what follows, recall the notation $\delta_Y\coloneqq\eq{Y}{\tmn}\top_{\P(Y)}$ from \cref{n:delta}.
In the proofs below we use the equivalent formulations of equality established in Appendix~\ref{appendix}.

\begin{lemma}\label{l:delta-id-times-delta}
Let $\P$ be a Boolean hyperdoctrine over $\C$ and let $X,Y\in\C$. Then the pullback square
\[
    \begin{tikzcd}[column sep = 6.5em]
        Y\times X  \arrow{r}{\id_Y\times\Delta_X}\arrow["\id_Y\times\Delta_X"']{d} & Y\times X \times X  \arrow["\id_Y\times\Delta_X\times \id_X"]{d}\\
        Y\times X \times X \arrow["\id_Y\times\id_X\times \Delta_X"']{r}& Y\times X \times X \times X
        \arrow["\lrcorner"{anchor=center, pos=0.125}, draw=none, from=1-1, to=2-2]
    \end{tikzcd}
\]
satisfies the Beck--Chevalley condition.
\end{lemma}
\begin{proof}
We shall prove
    \[
        \P(\id_Y\times\Delta_X\times\id_X)\circ \exists_{\id_Y\times\id_X\times\Delta_X} = \exists_{\id_Y\times\Delta_X} \circ \P(\id_Y\times\Delta_X).
    \]
    Let $\alpha\in\P(Y\times X\times X)$; we have
    \begin{align*}
        &\P(\id_Y\times\Delta_X\times\id_X)( \exists_{\id_Y\times\id_X\times\Delta_X}\alpha)\\
        &=\P(\id_Y\times\Delta_X\times\id_X)\big( \P(\pr^{Y\times X\times X\times X}_{1,2,3})(\alpha)\land \P(\pr^{Y\times X\times X\times X}_{3,4})(\delta_X)\big)&&\text{(by \eqref{i_eq:1}$\Rightarrow$\eqref{i_eq:2} in Thm.~\ref{t:defs_equality})}\\
        &=\P(\pr_{1,2}^{Y\times X\times X}) \big(\P(\id_Y\times\Delta_X)(\alpha)\big)\land \P(\pr_{2,3}^{Y\times X\times X})(\delta_X)\\
        &=\exists_{\id_Y\times\Delta_X}  \P(\id_Y\times\Delta_X)(\alpha)&&\text{(by \eqref{i_eq:1}$\Rightarrow$\eqref{i_eq:2} in Thm.~\ref{t:defs_equality})}
    \end{align*}
    as desired.
\end{proof}

\begin{lemma}\label{l:id-delta}
    Let $\P$ be a Boolean hyperdoctrine over $\C$ and let $X,Y\in\C$. Then the pullback square
    \[
        \begin{tikzcd}
            Y \times X  \arrow{r}{\id_{Y\times X}}\arrow[swap]{d}{\id_{Y\times X}} &  Y \times X  \arrow{d}{\id_Y\times\Delta_X}\\
            Y \times X  \arrow[swap]{r}{\id_Y\times\Delta_X}& Y \times X \times X
            \arrow["\lrcorner"{anchor=center, pos=0.125}, draw=none, from=1-1, to=2-2]
        \end{tikzcd}
    \]
    satisfies the Beck--Chevalley condition.
\end{lemma}

\begin{proof}
    We shall prove
        \[
            \P(\id_Y\times\Delta_X)\circ \exists_{\id_Y\times\Delta_X} = \exists_{\id_{Y\times X}} \circ \P(\id_{Y\times X}), 
        \]
    i.e., since the left adjoint to the identity on $\P(Y\times X)$ is the identity, and by functoriality of $\P$,    
        \[
            \P(\id_Y\times\Delta_X)\circ \exists_{\id_Y\times\Delta_X} = \id_{\P(Y\times X)}. 
        \]
    Let $\alpha\in\P(Y\times X)$; we have
        \begin{align*}
           &\P(\id_Y\times\Delta_X)( \exists_{\id_Y\times\Delta_X} \alpha)\\
           &=\P(\id_Y\times\Delta_X)\big( \P(\pr_{1,2}^{Y\times X\times X}) (\alpha)\land \P(\pr_{2,3}^{Y\times X\times X})(\delta_X)\big)&&\text{(by \eqref{i_eq:1}$\Rightarrow$\eqref{i_eq:2} in Thm.~\ref{t:defs_equality})}\\
            &=\alpha\land \P(\pr_{X}^{Y\times X})\big(\P(\Delta_X)(\delta_X)\big)\\
            &=\alpha&&\text{(by \eqref{i_eq:1}$\Rightarrow$\eqref{i:def=1} in Thm.~\ref{t:defs_equality})},
        \end{align*}
    as desired.
\end{proof}

\begin{remark}[Pasting Beck--Chevalley squares]\label{r:pasting}
    Pasting pullback squares satisfying the Beck--Chevalley condition (either horizontally or vertically) yields a composite pullback square that also satisfies the Beck--Chevalley condition.
\end{remark}

\subsection{Every Boolean hyperdoctrine over \texorpdfstring{$\FinSet\op$}{FinSet-op} is full}\label{ss:bool-hyp-full}

\begin{theorem}[Fullness over $\FinSet\op$]\label{t:all-full-finset}
    Every Boolean hyperdoctrine $\P \colon \FinSet \to \BA$ over $\FinSet\op$ is full.
    That is, for every pushout square in $\FinSet$ (on the left), the square in $\Pos$ on the right commutes.
    \[\begin{tikzcd}
        	Z & X & {\P(X)} & {\P(Z)} \\
        	Y & W & {\P(W)} & {\P(Y)}
        	\arrow["f", from=1-1, to=1-2]
        	\arrow["g"', from=1-1, to=2-1]
        	\arrow["\ulcorner"{anchor=center, pos=0.875}, draw=none, from=1-1, to=2-2]
        	\arrow["u", from=1-2, to=2-2]
        	\arrow["{\exists_f}", from=1-3, to=1-4]
        	\arrow["{\P(u)}"', from=1-3, to=2-3]
        	\arrow["{\P(g)}", from=1-4, to=2-4]
        	\arrow["v"', from=2-1, to=2-2]
        	\arrow["{\exists_v}"', from=2-3, to=2-4]
        \end{tikzcd}\]
\end{theorem}
\begin{proof}
    Let us consider the epi--mono factorization of the morphisms $f$ and $g$.
    \[\begin{tikzcd}
	Z && X & Z && Y \\
	& \bullet &&& \bullet
	\arrow["f", from=1-1, to=1-3]
	\arrow["{e_f}"', two heads, from=1-1, to=2-2]
	\arrow["g", from=1-4, to=1-6]
	\arrow["{e_g}"', two heads, from=1-4, to=2-5]
	\arrow["{m_f}"', hook, from=2-2, to=1-3]
	\arrow["{m_g}"', hook, from=2-5, to=1-6]
    \end{tikzcd}\]
    The pushout square on the left-hand side in the statement is the outer square of the following pasting of pushout squares in $\FinSet$:
    \begin{equation}\label{eq:4-pushout}
        \begin{tikzcd}[row sep = 3em, column sep = 3em]
        	Z & \bullet & X \\
        	\bullet & \bullet & \bullet \\
        	Y & \bullet & W
        	\arrow[two heads, from=1-1, to=1-2, "e_f"description]
        	\arrow["f"{description}, curve={height=-12pt}, from=1-1, to=1-3]
        	\arrow[two heads, from=1-1, to=2-1, "{e_g}"{name=0, description, anchor=center, inner sep=0}]
        	\arrow["\ulcorner"{anchor=center, pos=0.875}, draw=none, from=1-1, to=2-2]
        	\arrow["g"{description}, curve={height=12pt}, from=1-1, to=3-1]
        	\arrow[hook, from=1-2, to=1-3, "m_f"description]
        	\arrow[two heads, from=1-2, to=2-2, ""{name=1, anchor=center, inner sep=0}]
        	\arrow[two heads, from=1-3, to=2-3, ""{name=4}]
        	\arrow["u"{description}, curve={height=-12pt}, from=1-3, to=3-3]
        	\arrow[two heads, from=2-1, to=2-2]
        	\arrow[hook, from=2-1, to=3-1, "{m_g}"{name=2, anchor=center, inner sep=0, description}]
        	\arrow["\ulcorner"{anchor=center, pos=0.875}, draw=none, from=2-1, to=3-2]
        	\arrow[hook, from=2-2, to=3-2, ""{name=3, anchor=center, inner sep=0}]
        	\arrow["\ulcorner"{anchor=center, pos=0.875}, draw=none, from=2-2, to=3-3]
        	\arrow[hook, from=2-2, to=2-3]
        	\arrow[hook, from=2-3, to=3-3, ""{name=5}]
        	\arrow[two heads, from=3-1, to=3-2]
        	\arrow["v"{description}, curve={height=12pt}, from=3-1, to=3-3]
        	\arrow[hook, from=3-2, to=3-3]
        	\arrow["{(1)}"{description}, draw=none, from=0, to=1]
        	\arrow["{(3)}"{description}, draw=none, from=2, to=3]
        	\arrow["(2)"{description}, draw=none, from=1, to=4]
        	\arrow["(4)"{description}, draw=none, from=3, to=5]
        	\arrow["\ulcorner"{anchor=center, pos=0.875}, draw=none, from=1-2, to=2-3]
        \end{tikzcd}
    \end{equation}
    Indeed, in any category, the pushout of an epimorphism is an epimorphism, and moreover, in $\FinSet$, the pushout of an injection is an injection.
    
    By \cref{r:pasting}, it is enough to show that the four diagrams composing the pushout of $f$ along $g$ satisfy the Beck--Chevalley condition. Note that, by \cref{p:inverse-bc}, we may always choose whichever orientation of the pushout is most convenient. 
    
    For the squares $(2)$, $(3)$ and $(4)$, it is enough to show that any pushout of an injective function along any morphism satisfies the Beck--Chevalley condition. 

    Let $i \colon A \hookrightarrow B$ be an injective function in $\FinSet$; the injection $i$ is (up to isomorphism) a coproduct inclusion $i\colon A\hookrightarrow A\sqcup B'$. Moreover, the pushout of the inclusion $i$ along a function $k\colon A\to C$ is given by the following square.
    \[
    \begin{tikzcd}
    	A & {A\sqcup B'} \\
    	C & {C\sqcup B'}
    	\arrow["i", hook, from=1-1, to=1-2]
    	\arrow["k"', from=1-1, to=2-1]
    	\arrow["\ulcorner"{anchor=center, pos=0.875}, draw=none, from=1-1, to=2-2]
    	\arrow[from=1-2, to=2-2]
    	\arrow[hook, from=2-1, to=2-2]
    \end{tikzcd}
    \]
    According to property \eqref{i:h2} of \cref{d:bool_hyp} (Beck--Chevalley for existential), the pushout square above satisfies the Beck--Chevalley condition; indeed, it is a pullback square in $\FinSet\op$ of the form \cref{r:pullback}\eqref{eq:diagram} (with $n=0$).
    
    For the square $(1)$, let us decompose the pushout of $e_f$ along $e_g$ in a grid of pushout squares where every involved morphism is a pair-identifying map or an isomorphism as in \cref{l:surj}. By \cref{r:pasting}, it is enough to show that every pushout of a span of two pair-identifying maps satisfies the Beck--Chevalley condition. Indeed, the case where one of the involved maps is an isomorphism reduces to the case treated above (where one of the maps is injective).
    
    The pushout of a span of two pair-identifying maps has one of the three forms in \cref{r:pair-id}.

    The pushout in \cref{r:pair-id}\eqref{eq:pi-1} is a pullback square in $\FinSet\op$ of the form described in \cref{r:pullback}\eqref{eq:diagram} (with $n=2$), and thus by \cref{d:bool_hyp}\eqref{i:h4} (Beck--Chevalley for equality) it satisfies the Beck--Chevalley condition. To conclude, the pushout squares in \cref{r:pair-id}\eqref{eq:pi-2} and \eqref{eq:pi-3} are pullback squares in $\FinSet\op$ of the form described in \cref{l:delta-id-times-delta} and \cref{l:id-delta} respectively, and thus they satisfy the Beck--Chevalley condition.
\end{proof}

\begin{remark}[The syntactic hyperdoctrine of a theory in a relational language is full]\label{r:syn-full}
    Let $\T$ be a theory in a one-sorted relational language with equality, and let us consider its syntactic hyperdoctrine $\LT^\T \colon \Ctx_\varnothing\op  \to \BA$ (as in \cref{ex:synt-doc}). The category $\Ctx_\varnothing$ is equivalent to $\FinSet\op$. Indeed, since there are no function symbols, the only terms in context $X$ are the variables themselves, and so $\mathrm{Term}_\varnothing(X) = X$; thus a morphism $X \to Y$ in $\Ctx_\varnothing$---i.e., a function $Y \to \mathrm{Term}_\varnothing(X) = X$---is simply a function $Y \to X$, i.e., a morphism $X \to Y$ in $\FinSet\op$.
    Then, by \cref{t:all-full-finset}, the syntactic hyperdoctrine $\LT^\T \colon \FinSet \to \BA$ of $\T$ is full.
\end{remark}

\subsection{The many-sorted case}
\label{ss:many-sorted}

Let $S$ be a set, thought of as a set of sorts.
The category of contexts of an $S$-sorted empty functional language can be defined by adapting the construction in \cref{ex:synt-doc} to the many-sorted setting\footnote{Objects are finite sets of \emph{sorted} variables, and a morphism $\sigma\colon X\to Y$ maps every $y\in Y$ of sort $\mathsf S$ to a term in $\mathrm{Term}_\F(X)$ of the same sort $\mathsf S$.}.
Similarly to the one-sorted case (\cref{r:syn-full}), this category of contexts admits a concrete description as the opposite of the category $\FinSet_S$ of ``$S$-sorted finite sets'', which we next describe.
Let $\FinSet_S$ be the category whose objects are pairs $(X,\sigma)$, where $X$ is a finite set (whose elements are to be thought of as variables) and $\sigma\colon X\to S$ is a function (to be thought of as an assignment of a sort to each variable), and whose morphisms $(X,\sigma)\to(Y,\tau)$ are functions $h\colon X\to Y$ making the following diagram commute (i.e., substitutions that preserve sorts):
\[
\begin{tikzcd}
    X \arrow[rr,"h"] \arrow[dr,"\sigma"'] &&
    Y \arrow[dl,"\tau"] \\
    & S &
\end{tikzcd}
\]
Equivalently, $\FinSet_S$ is the full subcategory of $\Set/S$ consisting of maps with finite domain.\footnote{Note that $\FinSet_S$ is the free finite-coproduct completion of the discrete category on $S$, obtained by restricting the construction of abstract families to finite families; see \cite[\S~6.1]{BorceuxJanelidze2001}.}
To see that the category of contexts is indeed equivalent to $\FinSet_S\op$, note that, since the language has no function symbols, the terms in a context are precisely its variables, with their assigned sorts; moreover, a substitution from a context $(X,\sigma)$ to a context $(Y,\tau)$ assigns to each variable in $Y$ a variable in $X$ of the same sort, and hence corresponds to a morphism $(Y,\tau)\to(X,\sigma)$ in $\FinSet_S$.

\begin{proposition}[Fullness in the many-sorted case]
\label{p:all-full-many-sorted}
    Let $S$ be a set.
    Every Boolean hyperdoctrine over $\FinSet_S\op$ is full.
\end{proposition}

\begin{proof}
    The proof is the same as the one of \cref{t:all-full-finset}, keeping track of sorts.
    The epi--mono factorization and pushouts in $\FinSet_S$ are computed on the underlying finite sets, with the induced sort assignment; moreover, the pushout of an injection is still an injection.
    Hence the pasting of the four pushout squares in \eqref{eq:4-pushout} is available in $\FinSet_S$ as well.
    
    An injection is still, up to isomorphism, a coproduct inclusion. Thus, the squares $(2)$, $(3)$ and $(4)$ are again of the form \cref{r:pullback}\eqref{eq:diagram} with $n=0$, and hence they satisfy the Beck--Chevalley condition.
    
    Observe that, in $\FinSet_S$, a pair-identifying map identifies two variables of the same sort, and every surjection decomposes into pair-identifying maps and isomorphisms as in \cref{r:decomposition_surj}; moreover, the pushout of a span of two pair-identifying maps has one of the three shapes described in \cref{r:pair-id}, keeping track of the sorts. Thus, we can decompose the square $(1)$ into a grid as in \cref{l:surj}.
  
     The squares in the grid either involve isomorphisms, or are again of one of the forms \cref{r:pullback}\eqref{eq:diagram} with $n=2$, \cref{l:delta-id-times-delta} or \cref{l:id-delta}. Hence they satisfy the Beck--Chevalley condition.
\end{proof}

\begin{corollary}[Fullness for relational theories]
\label{c:syn-full}
    Let $\T$ be a first-order theory in a many-sorted
    purely relational language with equality.
    Its Lindenbaum--Tarski hyperdoctrine is full.
\end{corollary}

\appendix
\section{Equivalent definitions of ``having equality''}
\label{appendix}

We collect here some alternative equivalent definitions of ``having equality'' that appear in the literature.

\begin{theorem}[Equivalent formulations of equality]\label{t:defs_equality}
    Let $\C$ be a category with finite products and let $\P \colon \C\op \to  \BA$ be a functor. The following are equivalent.
    \begin{enumerate}
    \item\label{i_eq:1}
    \begin{enumerate}

        \item\label{i_eq:1_a} (Equality) For all $X,Y\in\C$, letting $\Delta_Y$ denote the diagonal morphism $\ple{\id_Y,\id_Y}\colon Y\to Y\times Y$, the function 
        \[
        \P(\id_X\times \Delta_Y)\colon \P(X\times Y\times Y)\to \P(X\times Y)
        \]
        has a left adjoint, denoted $\eq{Y}{X}$.
        
         \item\label{i_eq:1_b}
        (Beck--Chevalley for equality) For any morphism $f\colon X'\to X$ in $\C$ and every $Y\in\C$, the following square in $\Pos$ commutes.
        \[
            \begin{tikzcd}
            {X} & {\P(X\times Y)} & {\P(X\times Y\times Y)} \\
            X' & {\P(X'\times Y)} & {\P(X'\times Y\times Y)}
            \arrow["{\P(f\times\id_{Y})}", from=1-2, to=2-2, swap]
            \arrow["{\P(f\times \id_{Y\times Y})}", from=1-3, to=2-3]
            \arrow["{\eq{Y}{X'}}"', from=2-2, to=2-3]
            \arrow["{\eq{Y}{X}}", from=1-2, to=1-3]
            \arrow["f", from=2-1, to=1-1]
            \end{tikzcd}
        \]
        
    \end{enumerate}
    \item\label{i_eq:2} There is a family \[
    \big(\delta_Y\in\P(Y\times Y)\big)_{Y\in\C}
    \]
    such that, for every $X,Y\in\C$, the assignment
    \begin{align*}
        \P(X\times Y)&\longrightarrow \P(X\times Y\times Y)\\
        \alpha&\longmapsto \P(\pr^{X\times Y\times Y}_{1,2})(\alpha)\land \P(\pr^{X\times Y\times Y}_{2,3})(\delta_Y)
    \end{align*}    
    defines a left adjoint to
    $\P(\id_X\times \Delta_Y)\colon \P(X\times Y\times Y)\to \P(X\times Y)$.
    
    \item\label{i_eq:3} There is a family\footnote{Informally, condition \eqref{i:def=1} in \cref{t:defs_equality} is the reflexivity of the equality relation, i.e.\ $\vdash y=y$. Condition \eqref{i:def=2} is the substitutivity property, i.e.\ $\alpha(y)\land (y=y')\vdash \alpha(y')$. Finally, condition \eqref{i:def=3} can be roughly interpreted as ``$(x=x')\land (y=y')\vdash (x,y)=(x',y')$'', meaning that two pairs coincide if both entries do.}
    \[
    \big(\delta_Y\in\P(Y\times Y)\big)_{Y\in\C}
    \]
    such that, for all $X,Y\in\C$,
    \begin{enumerate}
        \item\label{i:def=1} (Reflexivity)
        denoting by $\Delta_Y \colon Y \to Y \times Y$ the diagonal $\ple{\id_Y , \id_Y}$, 
        \[
        \top_{\P(Y)}\leq\P(\Delta_Y)(\delta_Y) \quad \text{in $\P(Y)$};
        \]
        
        \item\label{i:def=2} (Substitutivity)
        denoting by $\pr_1, \pr_2 \colon Y\times Y \to Y$ the two projections, for every $\alpha\in\P(Y)$,
        \[
        \P(\pr_1)(\alpha)\land \delta_Y\leq \P(\pr_2)(\alpha)\quad \text{in $\P(Y\times Y)$};
        \]
        
        \item\label{i:def=3} (Product equality) denoting by $\pr_1, \pr_2, \pr_3,\pr_4$ the four projections from $X\times Y\times X\times Y$ to the factors,
        \[
        \Bigl(\P(\ple{\pr_1,\pr_3})(\delta_X)\Bigr)\land \Bigl(\P(\ple{\pr_2,\pr_4})(\delta_Y)\Bigr)\leq \delta_{X\times Y}\quad \text{in $\P(X\times Y\times X\times Y)$}.
        \]
    \end{enumerate}
    \end{enumerate}
\end{theorem}

\begin{proof}
    \eqref{i_eq:2} $\Leftrightarrow \eqref{i_eq:3}$. The proof can be found in \cite[Prop.~2.5]{EmPaRo20}.
    
    \eqref{i_eq:2} $\Rightarrow$ \eqref{i_eq:1}. Clearly, \eqref{i_eq:2} implies \eqref{i_eq:1_a}. Let us prove \eqref{i_eq:1_b}: for every $\alpha \in\P(X\times Y)$ we have
    \begin{align*}
    \P(f\times \id_{Y\times Y})(\eq{Y}{X}\alpha)&=\P(f\times \id_{Y\times Y})\big(\P(\pr^{X\times Y\times Y}_{1,2})(\alpha)\land \P(\pr^{X\times Y\times Y}_{2,3})(\delta_Y)\big)\\
    &=\P(\pr^{X'\times Y\times Y}_{1,2})\big(\P(f\times\id_{Y})(\alpha)\big)\land \P(\pr^{X'\times Y\times Y}_{2,3})(\delta_Y)\\
    &=\eq{Y}{X'}\P(f\times\id_{Y})(\alpha),
    \end{align*}
    as desired.
    
    \eqref{i_eq:1} $\Rightarrow$ \eqref{i_eq:2}. For every $Y\in\C$, set $\delta_Y\coloneqq\eq{Y}{\tmn}\top_{\P(Y)}\in\P(Y\times Y)$.
    For every $X,Y\in\C$, $\alpha\in\P(X\times Y)$ and $\beta\in\P(X\times Y\times Y)$, we have
    \begin{align*}
    &\P(\pr^{X\times Y\times Y}_{1,2})(\alpha)\land \P(\pr^{X\times Y\times Y}_{2,3})(\eq{Y}{\tmn}\top_{\P(Y)})\leq\beta \\
    &\iff \P(\pr_{2,3}^{X\times Y\times Y})(\eq{Y}{\tmn}\top_{\P(Y)})\leq  \P(\pr_{1,2}^{X\times Y\times Y})(\alpha)\to \beta \\
    &\iff \eq{Y}{X}\P(\pr_2^{X\times Y})(\top_{\P(Y)})\leq  \P(\pr_{1,2}^{X\times Y\times Y})(\alpha)\to \beta &&\text{(by BC wrt $!_X\colon X\to \tmn$)}\\
    &\iff \eq{Y}{X}\top_{\P(X\times Y)}\leq  \P(\pr_{1,2}^{X\times Y\times Y})(\alpha)\to \beta \\
    &\iff \top_{\P(X\times Y)}\leq \P(\id_X\times\Delta_Y)\big( \P(\pr_{1,2}^{X\times Y\times Y})(\alpha)\to \beta\big) &&\text{(since $\eq{Y}{X}\dashv\P(\id_X\times\Delta_Y)$)}\\
    & \iff \top_{\P(X\times Y)}\leq \alpha\to \P(\id_X\times \Delta_Y)(\beta)\\
    &\iff \alpha\leq\P(\id_X\times \Delta_Y)(\beta),
    \end{align*}
    as desired.
\end{proof}

\section{Beck--Chevalley is not preserved by multiplication by an object}
\label{appendix-B}

    In \cite[\S 8]{Seely83}, Seely characterized the pullback squares satisfying the Beck--Chevalley condition in terms of the derivability of two sequents in the associated theory.
    Seely's notion of hyperdoctrine differs from the one we consider here: in particular, the Beck--Chevalley condition is required to be preserved by multiplication by an object, i.e., if a pullback square $\mathscr{S}$ satisfies it, so does the pullback square obtained by multiplying $\mathscr{S}$ by a fixed object of the base category.
    In this appendix, we show that this requirement is not a consequence of the axioms of a Boolean hyperdoctrine over an arbitrary base category (\cref{p:BCC-not-stable-under-product}), so that Seely's characterization cannot be applied directly to obtain \cref{t:all-full-finset}\footnote{A posteriori, the requirement does hold for Boolean hyperdoctrines over $\FinSet\op$, as a consequence of \cref{t:all-full-finset}. Moreover, we point out that the proof of the \emph{sufficient} condition for the Beck--Chevalley condition in Seely's paper \cite[\S~8.2]{Seely83} seems not to use the closure under multiplication by an object.
    If this is the case, one may still attempt to prove that every Boolean hyperdoctrine $\FinSet \to \BA$ satisfies the two sequents in Seely's paper to obtain an alternative proof of its fullness.}:
    as a counterexample, we give the syntactic hyperdoctrine of a theory in a two-sorted language with a single function symbol.
\begin{definition}\label{d:theory-T}
    Let $\F$ be the two-sorted first-order functional language with sorts $\mathsf X,\mathsf Z$ and a unary function symbol $\mathsf f$ with input sort $\mathsf X$ and output sort $\mathsf Z$.

    Let $\T$ be the $\F$-theory axiomatized by the following axioms:
    \begin{enumerate}
        \item\label{i:ax-two} $\exists x_1\exists x_2\,\big(\lnot(x_1= x_2)\land\forall x\,(x= x_1\lor x= x_2)\big)$;
        \item\label{i:ax-nonempty} $\exists z\,\top$;
        \item\label{i:ax-one} $\forall z_1\forall z_2\,(z_1= z_2)$,
    \end{enumerate}
    where $x,x_1,x_2$ are $\mathsf X$-sorted variables and $z,z_1,z_2$ are $\mathsf Z$-sorted variables.
\end{definition}
    Informally, $\T$ says that the sort $\mathsf X$ has exactly two elements and that the sort $\mathsf Z$ has exactly one element.

\begin{remark}\label{r:completeness-T}
    Up to isomorphism, the theory $\T$ defined in \cref{d:theory-T} has a unique model $M$ given by $\mathsf X^M\coloneqq\{a_1,a_2\}$, $\mathsf Z^M\coloneqq\{*\}$ and $\mathsf f^M\coloneqq\,!$. In particular, $\T$ is consistent and complete, and so, for all formulas $\varphi,\psi$ in a context $Y$,
    \[
        \T \vdash \varphi\Rightarrow_Y\psi
        \quad\text{if and only if}\quad
        \varphi^M\subseteq\psi^M,
    \]
    where $(-)^M$ denotes the interpretation in $M$.
\end{remark}

\begin{remark}\label{r:aut-M}
    Let $M$ be the model defined in \cref{r:completeness-T}, and let $\tau\colon M\to M$ be the family of functions where the component $\tau^{\mathsf X}\colon\mathsf X^M\to\mathsf X^M$ is the switch of $a_1$ with $a_2$, and the component $\tau^{\mathsf Z}\coloneqq\id_{\mathsf Z^M}$.
    We observe that $\tau$ is a model homomorphism, since for every $a\in\mathsf X^M$
    \[
    \tau\big(\mathsf f ^M (a)\big)=\mathsf f ^M\big(\tau(a)\big).
    \]
    In particular, $\tau$ is an automorphism of $M$.
    
        Since isomorphisms of structures preserve satisfaction, the interpretation of a formula is invariant under the automorphisms of $M$; hence, for every formula $\alpha$ in context $\{x:\mathsf X\}$ and every $a\in\mathsf X^M$ we have
    \[
        a\in\alpha^M\iff\tau^{\mathsf X}(a)\in\alpha^M,
    \]
    i.e.\ $\tau^{\mathsf X}[\alpha^M]=\alpha^M$. Since $\tau^{\mathsf X}$ switches the two elements of $\mathsf X^M$, this forces
    \[
        \alpha^M=\varnothing
        \qquad\text{or}\qquad
        \alpha^M=\mathsf X^M.
    \]
\end{remark}

\begin{proposition}[Failure of stability of BC under $-\times T$]
\label{p:BCC-not-stable-under-product}
There are a Boolean hyperdoctrine $\P\colon\C\op\to\BA$, an object $T\in\C$ and a pullback square in $\C$ satisfying the Beck--Chevalley condition such that the pullback square obtained by applying $-\times T$ does not satisfy the Beck--Chevalley condition.
\end{proposition}

\begin{proof} 
    Let $\F$ and $\T$ be defined as in \cref{d:theory-T},
    and let $\Ctx_\F$ be the category of contexts as in \cref{ex:synt-doc}, adapted to many sorts.
    Let $\P\coloneqq\LT^\T\colon\Ctx_\F\op\to\BA$ be the syntactic hyperdoctrine of $\T$ as in \cref{ex:synt-doc} adapted to the many-sorted case. 
    Let $X\coloneqq\{x:\mathsf X\}$ and $Z\coloneqq\{z:\mathsf Z\}$, and let
    $\widehat{\mathsf f}\colon X\to Z$ be the morphism of $\Ctx_\F$ mapping $z:\mathsf Z$
    to the term $\mathsf f(x):\mathsf Z$, as in Subsection~\ref{ss:contexts-examples}.

    Using the argument of Subsection~\ref{ss:contexts-examples} adapted to the many-sorted case,
    we have that $\widehat{\mathsf f}$ is a monomorphism in $\Ctx_\F$. Hence, by
    \cref{r:mono-quasi}, the square
    \begin{equation}\label{diag:pb-2sort}
    \begin{tikzcd}
        X \arrow[r,"\id_{X}"]\arrow[d,"\id_{X}"'] & X\arrow[d,"\widehat{\mathsf f}"]\\
        X \arrow[r,"\widehat{\mathsf f}"'] & Z
        \arrow["\lrcorner"{anchor=center, pos=0.125}, draw=none, from=1-1, to=2-2]
    \end{tikzcd}
    \end{equation}
    is a pullback in $\Ctx_\F$.

    We shall prove that the Beck--Chevalley condition holds for \eqref{diag:pb-2sort}, i.e.\ that $\P(\widehat{\mathsf f})(\exists_{\widehat{\mathsf f}}\alpha)=\alpha$ for every $\alpha\in\P(X)$. By \cref{r:every-left-adj}, 
    $\exists_{\widehat{\mathsf f}}\alpha$ is the formula
    \[
    \exists x\,\big(z= \mathsf f(x)\land\alpha(x)\big)\qquad\text{in $\P(Z)$}.
    \]
    Thus, the Beck--Chevalley condition for \eqref{diag:pb-2sort} amounts to the provability from $\T$ of
    \begin{equation}\label{eq:bc-two-sorts}
        \exists x'\,\big(\mathsf f(x')= \mathsf f(x)\land\alpha(x')\big)
        \ \Leftrightarrow_{\{x\}}\
        \alpha(x).
    \end{equation}
       By \cref{r:completeness-T}, we can check the latter in $M$. The interpretation of the left-hand side is
    \[
        \{a\in\mathsf X^M\ \mid\ \text{there is }a'\in\mathsf X^M\text{ such that }\mathsf f^M(a')=\mathsf f^M(a)\text{ and }a'\in\alpha^M\}.
    \]
    Since $\mathsf f^M=\, !$, the condition $\mathsf f^M(a')=\mathsf f^M(a)$ holds for all $a,a'\in\mathsf X^M$, so the set above can be rewritten as
    \[
        \{a\in\mathsf X^M\ \mid\ \alpha^M\neq\varnothing\}=
        {\begin{cases}
        \varnothing&\text{if $\alpha^M=\varnothing$,}\\
        \mathsf X^M&\text{if $\alpha^M\neq\varnothing$.}
        \end{cases}}
    \]
    Thus, \eqref{eq:bc-two-sorts} holds because $\alpha^M$ is either $\varnothing$ or $\mathsf X^M$, as shown in \cref{r:aut-M}.

    Applying $-\times X$ to \eqref{diag:pb-2sort} gives the pullback square
    \[
    \begin{tikzcd}
        X\times X \arrow[r,"\id_{X\times X}"]\arrow[d,"\id_{X\times X}"'] & X\times X\arrow[d,"\widehat{\mathsf f}\times\id_X"]\\
        X\times X \arrow[r,"\widehat{\mathsf f}\times \id_X"'] & Z\times X
        \arrow["\lrcorner"{anchor=center, pos=0.125}, draw=none, from=1-1, to=2-2]
    \end{tikzcd}.
    \]
    We shall prove that the Beck--Chevalley condition does not hold for the square above.

    Let $X\times X=\{x_1:\mathsf X, x_2:\mathsf X\}$ and $Z\times X=\{z:\mathsf Z, x_2:\mathsf X\}$, so that $\widehat{\mathsf f}\times\id_X$ maps $z$ to $\mathsf f(x_1)$ and $x_2$ to $x_2$, and let $\beta\in\P(X\times X)$ be the (equivalence class of the) formula $x_1= x_2$. We show $\P(\widehat{\mathsf f}\times\id_X)(\exists_{\widehat{\mathsf f}\times\id_X}\beta)\neq\beta$.
    By \cref{r:every-left-adj}, $\exists_{\widehat{\mathsf f}\times\id_X}\beta$ is the formula
    \[
        \exists x_1'\exists x_2'(z= \mathsf f(x_1')\ \land\ x_2= x_2'\ \land\ x_1'= x_2')\qquad\text{in $\P(Z\times X)$},
    \]
    where $x'_1,x'_2$ are $\mathsf{X}$-sorted variables.
    Its substitution $\P(\widehat{\mathsf f}\times\id_X)(\exists_{\widehat{\mathsf f}\times\id_X}\beta)$ is the formula 
    \[
        \exists x_1'\exists x_2'(\mathsf f(x_1)= \mathsf f(x_1')\ \land\ x_2= x_2'\ \land\ x_1'= x_2')\qquad\text{in $\P(X\times X)$}, 
    \]
     whose interpretation in $M$ is the whole of $\mathsf X^M\times\mathsf X^M$: indeed, given $(a,a')\in\mathsf X^M\times \mathsf X^M$, take $x_1'= x_2'= a'$, so that the three conjuncts are satisfied because $\mathsf f^M(a)=*=\mathsf f^M(a')$; however, $\beta^M=\{(a_1,a_1),(a_2,a_2)\}$. By \cref{r:completeness-T}, the two classes are distinct, and so the Beck--Chevalley condition fails, as desired.
\end{proof}

\makeatletter
\begingroup
\let\addcontentsline\@gobblethree

\section*{Acknowledgments}
Marco Abbadini was funded by an FSR Incoming Postdoctoral Fellowship at the Université catholique de Louvain.
Francesca Guffanti was funded by the SHINE program of the French National Research Agency (ANR) under the project ``GULI'' (Grandeurs et Unités pour les Langages Informatiques), grant number ANR-22-EXES-0017.

\subsection*{AI use disclosure}\label{sec: AI}

We have been wondering for some time whether every Boolean hyperdoctrine is full. In response to our request for a counterexample, ChatGPT (OpenAI) proposed the hyperdoctrine assigning to each topological space the powerset of its set of connected components. When we asked how to find similar examples, it suggested looking for a functor to $\Set$ that preserves finite products and fails to preserve some pullback. We developed this idea into the general construction in \cref{s:non-full} and used it to obtain the indiscrete hyperdoctrine. We then asked whether fullness of the resulting hyperdoctrine was equivalent to pullback preservation; ChatGPT supplied a proof that fullness implies pullback preservation, which we revised for \cref{p:full-powerset-preserves-pullbacks}. Moreover, ChatGPT supplied the construction in \cref{ex:non-full-contexts} and the counterexample to stability of the Beck--Chevalley condition under multiplication by a fixed object in \cref{p:BCC-not-stable-under-product}.

\endgroup
\makeatother

\bibliographystyle{apalike}
\bibliography{Biblio}

\end{document}